\documentclass[11pt]{amsart}
\usepackage{amsmath}
\usepackage{mathrsfs}
\usepackage{yhmath}
\usepackage{amsxtra}
\usepackage{amscd}
\usepackage{amsthm}
\usepackage{amsfonts}
\usepackage{amssymb}
\usepackage{anysize}
\usepackage{bbding}
\usepackage{bbm}
\usepackage[bookmarks=true]{hyperref}
\usepackage{eucal}
\usepackage{enumerate}
\usepackage{fancyhdr}
\usepackage{mathtools,graphics,latexsym}
\usepackage{pifont}
\usepackage{pmgraph}
\usepackage{stmaryrd}
\usepackage[usenames,dvipsnames]{color}
\usepackage{tikz-cd}

\marginsize{3cm}{3cm}{3cm}{3cm}

\input prepictex
\input pictex
\input postpictex

\SetSymbolFont{stmry}{bold}{U}{stmry}{m}{n}

\newtheorem{thm}{Theorem}[section]
\newtheorem{cor}[thm]{Corollary}
\newtheorem{lem}[thm]{Lemma}
\newtheorem{prop}[thm]{Proposition}
\newtheorem{conj}[thm]{Conjecture}

\theoremstyle{definition}

\theoremstyle{remark}
\newtheorem{rem}[thm]{Remark}

\numberwithin{equation}{section}

\begin{document}

\newcommand{\thmref}[1]{Theorem~\ref{#1}}
\newcommand{\secref}[1]{Section~\ref{#1}}
\newcommand{\lemref}[1]{Lemma~\ref{#1}}
\newcommand{\propref}[1]{Proposition~\ref{#1}}
\newcommand{\corref}[1]{Corollary~\ref{#1}}
\newcommand{\conjref}[1]{Conjecture~\ref{#1}}
\newcommand{\remref}[1]{Remark~\ref{#1}}
\newcommand{\eqnref}[1]{(\ref{#1})}
\newcommand{\exref}[1]{Example~\ref{#1}}

\DeclarePairedDelimiterX\setc[2]{\{}{\}}{\,#1 \;\delimsize\vert\; #2\,}
\newcommand{\bn}[1]{\underline{#1\mkern-4mu}\mkern4mu }

\newcommand{\nc}{\newcommand}

\nc{\Z}{{\mathbb Z}}
\nc{\Zp}{\Z_+}
\nc{\C}{{\mathbb C}}
\nc{\N}{{\mathbb N}}
\nc{\D}{{\mathbb D}}

\nc{\hf}{\frac{1}{2}}
\nc{\bi}{\bibitem}
\nc{\wt}{\widetilde}
\nc{\dprime}{{\prime \prime}}
\nc{\ov}{\overline}
\nc{\un}{\underline}

\nc{\al}{\alpha}
\nc{\be}{\beta}
\nc{\de}{\delta}
\nc{\ep}{\epsilon}
\nc{\La}{\Lambda}
\nc{\la}{\lambda}
\nc{\si}{\sigma}
\nc{\ga}{\gamma}

\nc{\mc}{\mathcal}
\nc{\mf}{\mathfrak}
\nc{\cA}{\mc{A}}
\nc{\cC}{\mc{C}}
\nc{\cD}{\mc{D}}
\nc{\cE}{\mc{E}}
\nc{\cL}{{\mc L}}
\nc{\cP}{\mc{P}}
\nc{\cS}{\mc{S}}
\nc{\cT}{\mc{T}}
\nc{\fA}{{\mf A}}
\nc{\fB}{{\mf B}}
\nc{\fb}{{\mf b}}
\nc{\fg}{{\mf g}}
\nc{\fh}{{\mf h}}
\nc{\fS}{\mf{S}}

\nc{\bd}{\boldsymbol}
\nc{\wh}{\widehat}
\nc{\mr}{\mathring}
\nc{\Xl}{{\bd X}_{\! \ell}}
\nc{\Vac}{V_{{\rm crit}}}
\nc{\cUm}{\mc{U}_{-}}
\nc{\Ulz }{{U^\ell_{\! z}}}

\nc{\gl}{{\mf{gl}}}
\nc{\glmn}{\gl_{m|n}}
\nc{\glm}{\gl_{m}}
\nc{\hmn}{\fh_{m|n}}
\nc{\bmn}{\fb_{m|n}}
\nc{\bmns}{\cB^\s_{m|n}}
\nc{\hmns}{\fh^\s_{m|n}}
\nc{\Lmn}{L_{m|n}}
\nc{\Bmn}{\fB_{m|n}(\uz)}
\nc{\Bm}{\fB_{m}(\uz)}
\nc{\Bmnh}{\hat \fB_{m|n}(\uz)}
\nc{\smn}{\s_{m|n}}
\nc{\Smn}{\fS_{m+n}}
\nc{\cLmn}{\cL_{m|n}}
\nc{\Ximn}{\Xi_{m|n}}
\nc{\gmi}{t^{-1} \glmn[t^{-1}]}
\nc{\Ul}{U(\glmn)^{\otimes \ell}}
\nc{\vv}{|0\rangle}

\nc{\uz}{\un{\bd z}}
\nc{\uxi}{\un{\bd{\xi}}}
\nc{\uoxi}{\un{\bd{\ov\xi}}}
\nc{\ueta}{\un{\bd{\eta}}}

\nc{\s}{{\bf s}}
\nc{\hs}{\hat{s}}
\nc{\bs}{\bar{s}}
\nc{\bz}{\bar{0}}
\nc{\bo}{\bar{1}}

\nc{\T}{{\bf T}}
\nc{\pz}{\partial_z}
\nc{\Ber}{{\rm Ber}}
\nc{\Bers}{{\rm Ber}^{\bf s}}
\nc{\cdet}{{\rm cdet}}
\nc{\sing}{{\rm sing}}
\nc{\End}{{\rm End}}
\nc{\tr}{\mf{tr}}
\nc{\otr}{\ov{\mf{tr}}}

\nc{\glmrn}{\gl_{(m+r)|n}}
\nc{\hmrn}{\fh_{(m+r)|n}}
\nc{\Cmn}{{\mc C}_{m|n}}
\nc{\Cpk}{{\mc C}_{p|k}}
\nc{\Cmk}{{\mc C}_{m|k}}
\nc{\Cmrn}{{\mc C}_{(m+r)|n}}
\nc{\Imn}{\mathbb{I}_{m|n}}
\nc{\Eii}{E_{i,i}}
\nc{\mrn}{{(m+r)|n}}
\nc{\mrz}{{(m+r)|0}}
\nc{\mn}{{m|n}}
\nc{\Tmn}{\cT_{m|n}}

\nc{\fX}{{\mf X}^+}
\nc{\Xmn}{{{\mf X}^+_{m|n}}}
\nc{\Xmrn}{{{\mf X}^+_{(m+r)|n}}}
\nc{\OO}{\mathcal{O}}

\nc{\Pmn}{\mc{P}_{m|n}}
\nc{\Pmrn}{\mc{P}_{(m+r)|n}}
\nc{\ovla}{\ov{\la}}
\nc{\ovlamn}{{\ov{\la}^\mn}}

\nc{\fz}{{\mf z}}
\nc{\zmn}{\fz_{m|n}}
\nc{\zmnh}{\hat \fz_{m|n}}
\nc{\zgmn}{\fz(\wh{\gl}_{m|n})}
\nc{\zgm}{\fz(\wh{\gl}_m)}

\nc{\blb}{(\!(}
\nc{\brb}{)\!)}
\nc{\disp}{\displaystyle}
\nc{\ns}{\hspace{-0.3mm}}
\nc{\nns}{\hspace{-2mm}}
\nc{\tns}{\kern-1.1pt}

\advance\headheight by 2pt

\title[The Gaudin model and the Bethe ansatz]
{The Gaudin model for the general linear Lie superalgebra and the completeness of the Bethe ansatz}

\author[W. K. Cheong]{Wan Keng Cheong}
\address{Department of Mathematics, National Cheng Kung University, Tainan, Taiwan 701401}
\email{keng@ncku.edu.tw}

\author[N. Lam]{Ngau Lam}
\address{Department of Mathematics, National Cheng Kung University, Tainan, Taiwan 701401}
\email{nlam@ncku.edu.tw}

\begin{abstract}

Let $\Bmn$ be the Gaudin algebra of the general linear Lie superalgebra $\glmn$ with respect to a sequence $\uz \in \C^\ell$ of pairwise distinct complex numbers, and let $M$ be any $\ell$-fold tensor product of irreducible polynomial modules over $\glmn$. We show that the singular space $M^\sing$ of $M$ is a cyclic $\Bmn$-module and the Gaudin algebra $\Bmn_{M^\sing}$ of $M^\sing$ is a Frobenius algebra. We also show that $\Bmn_{M^\sing}$ is diagonalizable with a simple spectrum for a generic $\uz$ and give a description of an eigenbasis and its corresponding eigenvalues in terms of the Fuchsian differential operators with polynomial kernels. This may be interpreted as the completeness of a reformulation of the Bethe ansatz for $\Bmn_{M^\sing}$.

\end{abstract}

\maketitle

\setcounter{tocdepth}{1}

\section{Introduction}

The Gaudin model was introduced by Gaudin as a completely integrable quantum spin chain associated to the special linear algebra $\mf{sl}_2$ \cite{G76} and was later generalized to a model associated to an arbitrary semisimple Lie algebra $\mf{g}$ \cite{G83}.
Various generalizations of the model have since been proposed and investigated. Of particular interest is the Gaudin algebra (also known as the Bethe algebra) of $\mf{g}$ defined by (higher) Gaudin Hamiltonians for $\mf{g}$ (\cite{CF, FFR, MTV06, Ta}).
Although the general linear Lie algebra $\glm$ is not semisimple, the Gaudin algebra $\Bm$ for $\glm$, where $\uz:=(z_1, \ldots, z_\ell) \in \C^\ell$ is a sequence of pairwise distinct complex numbers, can be constructed in the same way as that for the special linear Lie algebra $\mf{sl}_m$.
Its structure and related Gaudin models have been explored in great detail (see, for instance, \cite{FFRyb, MTV06, MTV09-1, MTV09-2, MTV09-3, MTV10, MV05, MV07, Ryb}).

The problem of finding the common eigenvectors and eigenvalues of (higher) Gaudin Hamiltonians has played a central role in studying (generalized) Gaudin models.
Initially proposed by Bethe \cite{Bet} to find the eigenvectors and eigenvalues for the Hamiltonians of the XXX Heisenberg spin chain, the Bethe ansatz method has been extended to models associated to other spin chains and is widely used in statistical mechanics.

Let $V$ be an $\ell$-fold tensor product of finite-dimensional irreducible modules over $\glm$ and $V^\sing$ the singular space of $V$.
We denote by $\Bm_{V^\sing}$ the Gaudin algebra of $V^\sing$, which is defined to be the image of $\Bm$ in the algebra $\End(V^\sing)$ of linear endomorphisms of $V^\sing$.
The Bethe ansatz method explicitly describes a set of candidates for eigenvectors and the corresponding eigenvalues for $\Bm_{V^\sing}$.
The vectors obtained by the method are called Bethe vectors and are labeled by the solutions of the Bethe ansatz equations (\cite{BF, FFR}).
A famous conjecture predicts that the Bethe vectors form an eigenbasis for $\Bm_{V^\sing}$ for a generic $\uz$.
It is called the {\em completeness of the Bethe ansatz}. While the completeness is true for several examples (\cite{MV05}), counterexamples are found in \cite{MV07}.

 The obstacles encountered in the Bethe ansatz method can be overcome by the method of separation of variables developed by Sklyanin \cite{Sk}, which provides a construction of eigenvectors for the Gaudin model associated to $\gl_2$ based on the Fuchsian differential operators of order 2.
 Mukhin, Tarasov, and Varchenko \cite{MTV09-1, MTV09-2, MTV09-3} extended Sklyanin's approach to obtain a correspondence between the eigenspaces for $\Bm_{V^\sing}$ and the Fuchsian differential operators of order $m$ with polynomial kernels (see \eqref{Fuc}), where $V$ is an $\ell$-fold tensor product of irreducible polynomial modules over $\glm$.
 They also link the eigenvalues of $\Bm_{V^\sing}$ to the coefficients of the differential operators.
This yields a version of geometric Langlands correspondence for Gaudin models (see \cite{CT, Fr95, MTV12}).
Remarkably, the eigenvectors obtained from the correspondence form an eigenbasis for $\Bm_{V^\sing}$ if $\uz$ is generic (\thmref{glm-complete}).
It is unknown whether the eigenbasis admits a description as explicit as the one predicted by the Bethe ansatz method.
However, the positive answer to the diagonalization is what we initially expected for the Bethe ansatz.
It is thus reasonable to view Mukhin--Tarasov--Varchenko's result as the completeness of a reformulation of the Bethe ansatz for $\Bm_{V^\sing}$ even though it is different from the original one.

The Gaudin models for Lie superalgebras have also gained much attention (\cite{CCL, ChL, HM, HMVY, KM, Lu22, Lu23, MVY}).
In this paper, we are interested in the Gaudin algebra $\Bmn$ of the general linear Lie superalgebra $\glmn$, which is a subalgebra of $\Ul$ depending on a sequence $\uz \in \C^\ell$ of pairwise distinct complex numbers.
The algebra $\Bmn$ can be constructed via the Feigin--Frenkel center $\zgmn$ and the Berezinians (\secref{Gaudin}).
Furthermore, $\Bmn$ is commutative (\cite{MR}) and acts on any $\ell$-fold tensor product of $\glmn$-modules and its singular space.

Let $M$ be any $\ell$-fold tensor product of irreducible polynomial $\glmn$-modules.
We denote by $\Bmn_{N}$ the Gaudin algebra of $N$ for any $\Bmn$-submodule $N$ of $M$.
We prove the following by applying the tools of odd reflections in the theory of Lie superalgebras and the properties of Berezinians.

\begin{thm}[\thmref{cyc+frob}] \label{main1}
 $M^\sing$ is a cyclic $\Bmn$-module, and $\Bmn_{M^\sing}$ is a Frobenius algebra for any sequence $\uz \in \C^\ell$ of pairwise distinct complex numbers.
\end{thm}

\thmref{main1} implies that the algebra $\Bmn_{M^\sing}$ has dimension equal to $\dim M^\sing$ and each of its eigenspace is one-dimensional. Every generalized eigenspace of $\Bmn_{M^\sing}$ is also a cyclic $\Bmn$-module (\corref{max}).

\begin{thm}[\thmref{diag}] \label{main2}
$\Bmn_{M^\sing}$ is diagonalizable with a simple spectrum for a generic $\uz$.
\end{thm}

\thmref{main2} establishes the diagonalization of $\Bmn_M$ as $M$ is a direct sum of irreducible polynomial modules over $\glmn$ and the action of $\Bmn$ commutes with that of $\glmn$ on $M$ for a generic $\uz$.
We also see that given any singular weight $\mu$ of $M$, every eigenbasis for the algebra $\Bmn_{M^\sing_{\mu}}$ can be obtained from some eigenbasis for $(\fB_r)_{V^\sing}$, where $r$ is some positive integer and $V$ is an $\ell$-fold tensor product of some irreducible polynomial modules over $\gl_r$. Note that $r$ and $V$ depend on $\mu$.
Should the Bethe ansatz be complete for $(\fB_r)_{V^\sing}$, we will obtain an eigenbasis for $\Bmn_{M^\sing_{\mu}}$, which is written in terms of Bethe vectors in $V^\sing$, as well as a description of the corresponding eigenvalues (\thmref{Bethe-glmn-eigenvalue}).

Further, Mukhin--Tarasov--Varchenko's result on the diagonalization of the Bethe algebras for the general linear Lie algebras enables us to obtain an eigenbasis for $\Bmn_{M^\sing}$ which corresponds to the Fuchsian differential operators of appropriate orders with polynomial kernels and prescribed singularities.
Meanwhile, the corresponding eigenvalues can also be expressed as the coefficients of the differential operators with an appropriate adjustment (\thmref{glmn-eigenbasis}).
Such connections may be understood as a super version of geometric Langlands correspondence for Gaudin models.

We view \thmref{main2}, together with the relationship between the eigenspaces and Fuchsian differential operators, as the completeness of a reformulation of the Bethe ansatz for $\Bmn_{M^\sing}$.

The Feigin--Frenkel center $\zgmn$, arising in the construction of the Gaudin algebra $\Bmn$, is a huge commutative subalgebra of the superalgebra $U(\gmi)$.
Even for $n=0$, $\zgm$ is a polynomial algebra in infinitely many variables.
However, it has a nice property: a complete set of Segal--Sugawara vectors exists (\cite{FF, GW, Ha, CF, CM}).
It is conjectured that a similar (but not identical) property should be satisfied by $\zgmn$ for $n>0$ (see \conjref{FFC-gen} and also \cite[Remark 3.4(ii)]{MR}).
Solving the conjecture would be challenging. We will not deal with it here.
Instead, we will see that our results show some hope that there should be a positive answer to the conjecture.

The paper is organized as follows.
In \secref{Pre}, we present some background material that will be used in later sections.
In \secref{GM}, we introduce the Gaudin algebra $\Bmn$ of $\glmn$ and discuss some basics of polynomial $\glmn$-modules. We also establish some fundamental properties of $\Bmn$ acting on the tensor product of irreducible polynomial $\glmn$-modules.
In \secref{main}, we prove the main results of the paper. Particularly, we demonstrate \thmref{main1} (\thmref{cyc+frob}) and \thmref{main2} (\thmref{diag}).
In \secref{BAFF}, we relate our results to the Bethe ansatz and Fuchsian differential operators (\thmref{Bethe-glmn-eigenvalue} and \thmref{glmn-eigenbasis}). We also make some remarks on the Feigin--Frenkel center $\zgmn$.

\bigskip
\noindent{\bf Notations.} Throughout the paper, the symbols $\Z$, $\N$ and $\Zp$ stand for the sets of all, positive and non-negative integers, respectively, the symbol $\C$ for the field of complex numbers, and the symbol $\Z_2:=\{\bz, \bo\}$ for the field of integers modulo 2. All vector spaces, algebras, tensor products, etc., are over $\C$.
{\bf We fix $m \in \N$ and $n \in \Zp$}.

\bigskip

\section{Preliminaries} \label{Pre}

In this section, we introduce the general linear Lie superalgebra $\glmn$ and review the notions of Berezinians and pseudo-differential operators.

\subsection{The general linear Lie superalgebra $\glmn$}  \label{basic}

Let
$$
\Imn=\{\, 1, \ldots, m\} \cup \left\{\hf, \ldots, n-\hf \right\},
$$
and let $\pi:  \{1, \ldots, m+n \}  \longrightarrow \Imn $ be the bijection defined by
$$
\pi(i)=
\begin{cases}
\hspace{3mm} i  &\quad \text{if } i=1, \ldots, m;\\
      i-(m+\hf)         &\quad\text{if } i=m+1, \ldots, m+n.
\end{cases}
$$
We endow $\Imn$ with the total order given by the usual order of $\{1, \ldots, m+n \}$ via $\pi$. More precisely,
$$1<\cdots<m<\hf<\cdots<n-\hf.$$
The parity of $i$ is defined to be $|i|:=\ov{2i} \in \Z_2$ for $i \in \Imn$.

Let $\{e_i \, | \, i \in \Imn\}$ be the standard homogeneous ordered basis for the superspace $\C^{\mn}$, where the parity of $e_i$ is given by $|e_i|=|i|$.
The superspace of all $\C$-linear endomorphisms on $\C^{\mn}$ has a natural Lie superalgebra structure, which we denote by $\glmn$.
For any $i, j \in \Imn$, we denote by $E_{i,j}$ the $\C$-linear endomorphism on $\C^{\mn}$ defined by $E_{i,j} (e_k)=\de_{j, k} e_i$ for $k \in \Imn$, where $\de$ denotes the Kronecker delta. The set $\{ E_{i,j} \, | \, i, j \in \Imn \}$ is a homogeneous basis for $\glmn$.

Let
$
\bmn=\bigoplus_{\substack{ i,j \in \Imn, i \le j} }  \C E_{i,j}
$
be the standard Borel subalgebra of $\glmn$.
The corresponding Cartan subalgebra $\hmn$ of $\glmn$ has an ordered basis $\{ \Eii \, | \, i \in \Imn \}$.
The ordered dual basis in $\hmn^{*}$ is denoted by the set $\{ \ep_i \, | \, i \in \Imn \}$, where the parity of $\ep_i$ is given by $|\ep_i|=|i|$.

We will drop the symbol $|0$ from the subscript $m|0$. For instance, $\gl_m:=\gl_{m|0}$, $\fh_m:=\fh_{m|0}$, etc.

\subsection{Berezinians}\label{Ber}

Let $\cA$ be an associative unital superalgebra over $\C$.
The parity of a homogeneous element $a \in \cA$ is denoted by $|a|$, which lies in $\Z_2$.

Fix $k \in \N$. Let $A$ be a $k \times k$ matrix over $\cA$.
For any nonempty subset $P=\{i_1<\ldots<i_p\}$ of $\{1,\ldots,k\}$, the matrix $A_{P}:=\big[a_{i,j}\big]_{i, j \in P}$ is called a \emph{standard submatrix} of $A$.

Assume that $A$ has a two-sided inverse $A^{-1}=\big[\wt{a}_{i,j}\big]$.
For all $i, j=1, \ldots, k$, the \emph{$(i,j)$th quasideterminant} of $A$ is defined to be $|A|_{ij}:=\wt{a}_{j,i}^{-1}$ provided that $\wt{a}_{j,i}$ has an inverse in $\cA$. Following the notation of \cite{GGRW}, we write
$$
|A|_{ij}=
\begin{vmatrix}
 a_{1,1} 	& \ldots & a_{1,j} &\ldots & a_{1,k}\\
 \ldots	& \ldots & \ldots  &\ldots & \ldots	\\
 a_{i,1}& \ldots & \fbox{$a_{i,j}$} &\ldots &a_{i,k}\\
 \ldots	& \ldots & \ldots  &\ldots & \ldots	\\
 a_{k,1}&\ldots &a_{m,j} &\ldots &a_{k,k}
\end{vmatrix}.
$$
For $i=1,\ldots,k$, we define
$$
d_i(A)=\begin{vmatrix}
	 a_{1,1} 	& \ldots & a_{1,i}\\
	 \ldots	& \ldots & \ldots \\
	 a_{i,1}& \ldots & \fbox{$a_{i,i}$}
\end{vmatrix},\
$$
which are called the \emph{principal quasiminors} of $A$.

A $k \times k$ matrix $A$ over $\cA$ is called \emph{sufficiently invertible} if every principal quasiminor of $A$ is well defined, and $A$ is called \emph{amply invertible} if each of its standard submatrices is sufficiently invertible.

Let ${\s}=(s_1,\ldots,s_{m+n})$ be a sequence of 0's and 1's such that exactly $m$ of the $s_i$'s are 0 and the others are 1.
We call such a sequence a \emph{$0^m1^n$-sequence}.
Every $0^m1^n$-sequence can be written in the form
$(0^{m_1}, 1^{n_1}, \ldots, 0^{m_r}, 1^{n_r}),$
where the sequence begins with $m_1$ copies of $0$'s, followed by $n_1$ copies of $1$'s, and so on.
The set of all $0^m1^n$-sequences is denoted by $\cS_{\mn}$.

Let $\fS_{m+n}$ be the symmetric group on $\{1, \ldots, m+n\}$, and let $\s \in \cS_{\mn}$.
For any $\si \in \fS_{m+n}$ and any $(m+n)\times (m+n)$ matrix $A=[a_{i,j}]$ over $\cA$, we define $\s^\si=\left(s_{\si^{-1}(1)},s_{\si^{-1}(2)},\ldots,s_{\si^{-1}(m+n)} \right)$ and $A^\si=\big[a_{\si^{-1}(i),\si^{-1}(j)} \big]$.
We say that $A$ is of type $\s$ if $a_{i,j}$ is a homogeneous element of parity $|a_{i,j}|=\bs_i+\bs_j$ for any $i,j = 1,\ldots,m+n$.

For any $(m+n)\times (m+n)$ sufficiently invertible matrix $A$ of type $\s$  over $\cA$, the \emph{Berezinian of type $\s$} of $A$ is defined to be
$$
\Bers A=d_1(A)^{\hs_1}\ldots d_{m+n}(A)^{\hs_{m+n}},
$$
where $\hs_i:=(-1)^{s_i}$ (see \cite[(3.3)]{HM}). We refer the reader to \cite{Ber, Na, MR} for earlier definitions of Berezinians. 

An $(m+n)\times (m+n)$ matrix $A=\big[a_{i,j}\big]$ over $\cA$ is called a \emph{Manin matrix} of type $\s$ if $A$ is a matrix of type $\s$ satisfying the following relations
\begin{equation}
[a_{i,j}, a_{k,l}]=(-1)^{s_i s_j+s_i s_k+s_j s_k}[a_{k,j}, a_{i, l}]
\qquad\text{for all}\quad i,j,k,l = 1,\ldots,m+n,
\end{equation}
where $[a, b]:=ab-(-1)^{|a||b|}ba$ for any homogeneous elements $a, b \in \cA$.
The proof of the following proposition is straightforward (cf. \cite[Section 3]{HM}).

\begin{prop} 

Let $A$ be an $(m+n)\times (m+n)$ Manin matrix of type $\s$ over $\cA$.

\begin{enumerate}[\normalfont(i)]

\item If $P =\{i_1<\ldots<i_p\}$ is a nonempty subset of $\{1,\ldots,m+n\}$, then the standard submatrix $A_{P}$ of $A$ is a Manin matrix of type ${\s}_P:=(s_{i_1},\ldots s_{i_p})$.

\item For any $\si \in \fS_{m+n}$, $A^\si$ is a Manin matrix of type $\s^\si$.

\end{enumerate}

\end{prop}

The following propositions are useful in calculating the Berezinians of Manin matrices.

\begin{prop}[{\cite[Proposition 3.5]{HM}}] \label{HM1}
Let $A$ be an $(m+n)\times (m+n)$ amply invertible Manin matrix of type $\s$ over $\cA$.
Fix $k\in \{1,\ldots,m+n-1\}$. We write
$$
A=\begin{bmatrix} W & X\\ Y&Z \end{bmatrix},
$$
where $W,X,Y,Z$ are respectively $k\times k$, $k\times(m+n-k)$, $(m+n-k)\times k$, and $(m+n-k)\times (m+n-k)$ matrices.
Then the matrices $W$ and $Z-YW^{-1}X$ are sufficiently invertible Manin matrices of types $\s^\prime:=(s_1,\ldots,s_k)$ and $\s^\dprime:=(s_{k+1},\ldots,s_{m+n})$, respectively. Moreover,
$$
\Bers A={\rm Ber}^{\s^\prime} W \cdot {\rm Ber}^{\s^\dprime} \big(Z-YW^{-1}X \big).
$$
\end{prop}

\begin{prop} [{\cite[Proposition 3.6]{HM}}] \label{HM2}

Let $A$ be an $(m+n)\times (m+n)$ amply invertible Manin matrix of type $\s$ over $\cA$.
Then
    $$
    {\rm Ber}^{\s^\si} A^\si=\Bers A  \qquad \mbox{for any $\si \in \fS_{m+n}$.}
    $$ 
\end{prop}

Let $A$ be an $m \times m$ matrix over $\cA$. The \emph{column determinant} of $A$ is defined to be
$$
\cdet \,  A=\sum_{\si \in \mf{S}_m} (-1)^{l(\si)} \,  a_{\si(1),1}\ldots a_{\si(m),m},
$$
where $l(\si)$ denotes the length of $\si$. We have the following.

\begin{prop}[{\cite[Lemma 8]{CFR}}] \label{CFR}
Let $\s=(0^m)$. For any $m\times m$ sufficiently invertible Manin matrix $A$ of type $\s$ over $\cA$, we have $\Bers A=\cdet \,  A$.
\end{prop}

\propref{CFR} has the following generalization.

\begin{prop}[{\cite[Definition 2.6 and Theorem 2.11]{MR}}] 
Let $\s=(0^m, 1^n)$.
For any $(m+n)\times (m+n)$ sufficiently invertible Manin matrix $A$ of type $\s$ over $\cA$, we have 
$$
\Bers A=\Big( \sum_{\si \in \mf{S}_m} (-1)^{l(\si)} \,  a_{\si(1),1}\ldots a_{\si(m),m} \Big) \Big( \sum_{\tau \in \mf{S}_n} (-1)^{l(\tau)} \,  \wt{a}_{m+1, m+\tau(1)}\ldots \wt{a}_{m+n, m+\tau(n)} \Big),
$$
where $\wt{a}_{i,j}$ denotes the $(i,j)$th entry of $A^{-1}$.
\end{prop}

\subsection{Pseudo-differential operators} \label{pseudo}

Let $\cA$ be an associative unital superalgebra over $\C$ and $z$ an even variable.
We denote by $\cA [\tns[z]\tns]$ and $\cA \blb z \brb$ the superalgebras of formal power series and Laurent series in $z$ with coefficients in $\cA$, respectively. Let $\cA\blb z^{-1}, \pz^{-1} \brb$ denote the set of all formal series of the form
$$
  \sum_{j=-\infty}^s \sum_{i=-\infty}^r a_{ij} z^i \pz^j,
$$
for some $r, s \in \Z$ and $a_{ij} \in \cA$. We may endow $\cA\blb z^{-1}, \pz^{-1} \brb$ with a superalgebra structure using the rules:
\begin{equation}\label{rule}
\pz \pz^{-1}=\pz^{-1} \pz=1, \quad\pz^{j} z^k=\sum_{i=0}^{\infty}  \binom{j}{i} \binom{k}{i} i! \,   z^{k-i} \, \pz^{j-i},  \quad \mbox{for $j, k \in \Z$.}
\end{equation}
Here $\disp{\binom{j}{i}:=\frac{j(j-1)\ldots (j-i+1)}{i!}}$.
We refer $\cA\blb z^{-1}, \pz^{-1} \brb$ to as the superalgebra of pseudo-differential operators over $\cA$.
For any $\disp{f:=\sum_{j=-\infty}^s \sum_{i=-\infty}^r a_{ij} z^i \pz^j \in \cA\blb z^{-1}, \pz^{-1} \brb}$ and any $\cA$-module $M$, we define
\begin{equation}\label{pdiff}
f(v)=\sum_{j=-\infty}^s \sum_{i=-\infty}^r a_{ij}(v) z^i \pz^j, \qquad\hbox{for all $v\in M$.}
\end{equation}

\section{The Gaudin model of $\glmn$} \label{GM}

In this section, we introduce the Feigin--Frenkel center and the Gaudin algebra for $\glmn$. We discuss the basics of polynomial $\glmn$-modules required in this paper. Moreover, we establish some fundamental properties of the action of the Gaudin algebra of $\glmn$ on the tensor product of irreducible polynomial $\glmn$-modules.

\subsection{The Feigin--Frenkel center} \label{FFC}
For any Lie superalgebra $\fg$, we denote by $U(\fg)$ the universal enveloping algebra of $\fg$.
The loop algebra $\fg[t, t^{-1}]:=\fg\otimes \C[t, t^{-1}]$, where $t$ is an even variable, is a Lie superalgebra with the bracket given by
$$
\big[A_1[r_1], A_2[r_2] \big]=[A_1, A_2] [r_1+r_2] \qquad \mbox{for $A_1, A_2 \in \fg$ and $r_1, r_2 \in \Z$.}
$$
Here $A_i[r_i]:=A_i \otimes t^{r_i}$, and $[A_1, A_2]$ is the supercommutator of $A_1$ and $A_2$.

The general linear Lie superalgebra $\glmn$ is equipped with an invariant symmetric bilinear form $(\cdot, \cdot)$ defined by
$$
(A_1, A_2)=(n-m){\rm Str}(A_1 A_2)+ {\rm Str}(A_1) {\rm Str}(A_2) , \quad \mbox{for $A_1, A_2 \in \glmn$}.
$$
Here ${\rm Str}$ stands for the supertrace, which is defined by ${\rm Str}(E_{i,j})=(-1)^{2i} \de_{i,j}$ for $i, j \in \Imn$.
The affine Lie superalgebra
$$
\wh\gl_{m|n}:=\gl_{m|n}[t,t^{-1}]\oplus \C K,
$$
where $K$ is even and central, is a Lie superalgebra with the bracket given by
$$
\big[A_1[r_1], A_2[r_2] \big]=[A_1, A_2] [r_1+r_2] + r_1 \de_{r_1+r_2, 0} (A_1, A_2) K
$$
for $A_1, A_2 \in \glmn$ and $r_1, r_2 \in \Z$.

The vacuum module $\Vac(\glmn)$ at the critical level is
$$
\Vac(\glmn):=U(\wh\gl_{m|n})/I,
$$
where $I$ is the left ideal of $U(\wh\gl_{m|n})$ generated by $\glmn[t]$ and $K-1$ (see \cite[Section 1.3]{MR}).
Let $V=\Vac(\glmn)$. There is a unique vertex algebra structure on $V$ such that
the vacuum vector is $\vv:=1+ I \in V$,
the translation operator $T \in \End(V)$ is defined by the relations
$$T\vv=0, \qquad \big[T, A[r] \big]=-r A[r-1], \quad \mbox{for $A \in \glmn$ and $r \in \Z$,}$$
where $A[r]$ is considered as an element of $\End(V)$ which acts on $V$ by left multiplication, and
the state-field correspondence $Y(\cdot, z): V  \longrightarrow  \End (V)  [\tns[z, z^{-1}]\tns]$ is defined by
$$
Y(\vv, z)=1, \qquad
Y(A[-1] \vv, z)=\sum_{r \in \Z} A[r]z^{-r-1} , \quad \mbox{for $A \in \glmn$,}
$$
and is extended to all of $V$ by means of the reconstruction theorem (see \cite[Theorem 4.5]{Kac} or \cite[Theorem 4.4.1]{FBZ}).
The vertex algebra $V$ is called the {\em universal affine vertex algebra} associated to $\glmn$ at the critical level.

The center of the vertex algebra $V$ is given by
$$
\fz(\wh\gl_{m|n}):=\setc*{\! v \in V}{\glmn[t] v=0 \!},
$$
which is called the {\em Feigin--Frenkel center}. Any element of $\fz(\wh\gl_{m|n})$ is called a {\em Segal--Sugawara vector}.
It follows from the axioms of a vertex algebra that $\fz(\wh\gl_{m|n})$ is a commutative associative unital algebra and is $T$-invariant.
Let $$\cUm=U(\gmi).$$
By the Poincaré–Birkhoff–Witt theorem, there is a linear isomorphism from $V$ to $\cUm$, where $\vv \in V$ is mapped to $1 \in \cUm$.
As it restricts to an injective linear homomorphism $\fz(\wh\gl_{m|n}) \hookrightarrow \cUm$ which respects multiplication, we may view $\fz(\wh\gl_{m|n})$ as a subalgebra of $\cUm$.
We refer to \cite{Fr07, FBZ, Kac, MR} for further details on the vertex algebra $V$ and the Feigin--Frenkel center.

Let $\tau=-\partial_t$.
Similar to the superalgebra of pseudo-differential operators defined in \secref{pseudo}, we may consider the superalgebra $\cUm \blb \tau^{-1} \brb$ satisfying the rules \eqref{rule} (where $z$ and $\pz$ are replaced with $t$ and $\partial_t$, respectively).
Let $\smn=(0^m, 1^n) \in \cS_{\mn}$ and
$$\cT_{\mn}=\Big[\de_{i,j} \tau+(-1)^{2\pi(i)}E_{\pi(i),\pi(j)}[-1] \Big]_{i,j=1, \ldots, m+n}.$$
The matrix $\cT_{\mn}$ is an amply invertible Manin matrix of type $\smn$ over $\cUm \blb \tau^{-1} \brb$ (see \cite[Lemma 3.1]{MR}).
It is easy to see that $1+u\Tmn$ is an amply invertible Manin matrix of type $\smn$ over $\cUm [\tau] [\tns[u]\tns]$, where $u$ is an even variable.
Thus for $\si \in \Smn$, both $\Tmn^\si$ and $1+u\Tmn^\si$ are amply invertible Manin matrices of type $\smn^\si$.

We suppress the superscript $\smn$, e.g., $\Ber:=\Ber^{\smn}$.
The following is a direct consequence of \propref{HM2}.

\begin{prop} \label{perm-tau}
For each $\si \in \Smn$,
$$
\Ber^{\smn^\si} \! \left( 1+u\Tmn^\si \right)=\Ber  \left(1+u\Tmn \right).
$$
\end{prop}

The Berezinian $\Ber \left(1+u\Tmn \right)$ encodes a distinguished set of Segal--Sugawara vectors, as shown below.

\begin{prop} [{\cite[Corollary 3.3]{MR}}] \label{MR}
The Berezinian $\Ber \left(1+u\Tmn \right)$ has the expansion
\begin{equation} \label{Ber-tau-u}
\Ber \left(1+u\Tmn \right)=\sum_{i=0}^{\infty} \sum_{j=0}^i b_{ij} \tau^{i-j} u^i  \quad \mbox{for some $b_{ij} \in \zgmn$.}
\end{equation}
\end{prop}

Let $\zmn$ be the subalgebra of $\cUm$ generated by the elements $b_{ij}$, for $i,j \in \Zp$ with $j \le i$.
By \propref{MR}, $\zmn$ is a commutative subalgebra of $\zgmn$. It gives rise to the Gaudin algebra of $\glmn$ to be described in \secref{Gaudin}.

We may give another set of generators for $\zmn$ as follows. Define
$$
\Phi: \cUm \blb \tau^{-1} \brb \longrightarrow \cUm [\tau] \blb u \brb
$$
by
$$
\Phi \left(\sum_{i=-\infty}^r a_i \tau^i \right)=\sum_{i=-\infty}^r a_i (\tau+u^{-1})^i, \quad \mbox{for $a_i \in \cUm$ and $r \in \Z$}.
$$
Here $\disp{(\tau+u^{-1})^i:=\sum_{j=0}^\infty \binom{i}{j}  \tau^j u^{j-i}}$.
The map $\Phi$ is an injective superalgebra homomorphism, whose verification is straightforward and parallel to that of \cite[Lemma 4.1]{HM}.
We also have the expansion
\begin{equation} \label{Ber-tau}
\Ber \big(\Tmn \big)=\sum_{k=-\infty}^{m-n}b_k \tau^k \quad \mbox{for some $b_k \in \cUm$.}
\end{equation}

\begin{prop} \label{generator}
 The algebra $\zmn$ is generated by $b_k$ for $k \in \Z$ with $k \le m-n$. Moreover,
\begin{equation} \label{b_ij}
b_{ij}=\binom{m-n-j}{i-j} b_{m-n-j}, \qquad \mbox{for $i,j \in \Zp$ with $j \le i$}.
\end{equation}
\end{prop}

\begin{proof}
We have
\begin{eqnarray*}
\Ber  \! \left(1+u\Tmn \right) \ns \nns &=& \nns u^{m-n} \, \Ber \! \left(  \left[\de_{i,j} (\tau+u^{-1})+(-1)^{2\pi(i)}E_{\pi(i),\pi(j)}[-1] \right]_{i,j=1, \ldots, m+n}\right)\\
&=& \nns u^{m-n} \, \Ber \! \left(  \left[ \Phi \left(\de_{i,j} \tau+(-1)^{2\pi(i)}E_{\pi(i),\pi(j)}[-1] \right) \right]_{i,j=1, \ldots, m+n} \right)\\
&=& \nns u^{m-n} \, \Phi \! \left( \Ber \big(\Tmn \big) \right).
\end{eqnarray*}
Equating the coefficients of $\tau^{i-j} u^i$ in $\Ber  \! \left(1+u\Tmn \right)$ and $u^{m-n} \, \Phi \! \left( \Ber \big(\Tmn \big) \right)$ yields \eqref{b_ij} and proves the proposition.
\end{proof}

\begin{rem}
The equalities \eqref{b_ij} are known for non-super cases (see, for example, \cite[(7.12)]{Mo}).
\end{rem}

Analogous to \propref{perm-tau}, we have the following.

\begin{prop} \label{perm-tau-2}
For each $\si \in \Smn$,
$$
\Ber^{\smn^\si} \! \left(\Tmn^\si \right)=\Ber  \left(\Tmn \right).
$$
\end{prop}

\subsection{The Gaudin algebra of $\glmn$} \label{Gaudin}

Fix $\ell \in \N$. For any even variable $z$, let $$\Ulz =\Ul\blb z^{-1} \brb.$$
Let
$$
\Xl=\setc*{\! (z_1, \ldots,z_\ell)\in \C^\ell}{z_i\not=z_j \,\, \mbox{for any $i\not=j$} \!}
$$
be the \emph{configuration space} of $\ell$ distinct points on $\C^\ell$, and let $\uz \in \Xl$.
There is a superalgebra homomorphism
$$
 \Psi_{\tns \uz} : \cUm[\tau]  \longrightarrow \Ulz  [\pz]
$$
given by
 $$
  \Psi_{\tns \uz}(A[-r])=\sum_{i=1}^\ell \frac{A^{(i)}}{(z_i-z)^{r}}, \quad \mbox{for $A \in \glmn \,$ and $\, r \in \N$,}  \quad {\rm and} \quad  \Psi_{\tns \uz}(\tau)=\pz.
 $$
 Hereafter $A^{(i)}:=\underbrace{1\otimes\cdots\otimes1\otimes \stackrel{i}{A}\otimes1\otimes\cdots\otimes1}_{\ell}$ for $i=1, \ldots, \ell$, and every rational function in $z$ represents its power series expansion at $\infty$.
 The map $ \Psi_{\tns \uz}$ extends to a superalgebra homomorphism
$$
\tilde  \Psi_{\tns \uz} : \cUm[\tau] [\tns[u]\tns]  \longrightarrow \Ulz  [\pz] [\tns[u]\tns]
$$
by letting $\tilde  \Psi_{\tns \uz} (u)=u$.

Define
$$
A(z)=\sum_{i=1}^\ell \frac{A^{(i)}}{z-z_i} \in \Ulz , \qquad \mbox{for $A \in \glmn$}.
$$
We have $ \Psi_{\tns \uz}(A[-1])=-A(z)$ for $A \in \glmn$.
Consider the $(m+n) \times (m+n)$ matrix
\begin{eqnarray*}
\cLmn(\uz) \nns &:=& \nns \Big[  \Psi_{\tns \uz} \Big(\de_{i,j} \tau+(-1)^{2\pi(i)}E_{\pi(i),\pi(j)}[-1] \Big) \Big]_{i,j=1, \ldots, m+n}\\
\nns &\,\, =& \nns \Big[\de_{i,j} \pz-(-1)^{2\pi(i)}E_{\pi(i),\pi(j)}(z)\Big]_{i,j=1, \ldots, m+n},
\end{eqnarray*}
which is clearly an amply invertible Manin matrix of type $\smn$ over $\Ul \blb z^{-1}, \pz^{-1} \brb$.
We see that
$
\Ber \big( 1+u \cLmn(\uz) \big)=\tilde \Psi_{\tns \uz} \left(\Ber \big( 1+u\Tmn \big) \right).
$
In view of \eqref{Ber-tau-u},
$$
\Ber \big(1+u\cLmn(\uz) \big) =  \sum_{i=0}^{\infty} \sum_{j=0}^i b_{ij}(z) \pz^{i-j} u^i ,
$$
where $b_{ij}(z):=  \Psi_{\tns \uz}(b_{ij}) \in \Ulz$. The series $b_{ij}(z)$ are called {\em Gaudin Hamiltonians} for $\glmn$.

Let $\Bmn$ be the subalgebra of $\Ul$ generated by the coefficients of the Gaudin Hamiltonians $b_{ij}(z)$, for $i,j \in \Zp$ with $j \le i$.
By \cite[Corallary 3.6]{MR}, $\Bmn$ is a commutative subalgebra of $\Ul$.
The algebra $\Bmn$ is called the {\em Gaudin algebra} of $\glmn$.

The map $\Psi_{\tns \uz}$ also extends to a superalgebra homomorphism
$$\ov\Psi_{\tns \uz}: \cUm \blb \tau^{-1} \brb  \longrightarrow \Ul \blb z^{-1}, \pz^{-1} \brb$$
given by
$$
\ov\Psi_{\tns \uz} \left(\sum_{i=-\infty}^r a_i \tau^i \right)=\sum_{i=-\infty}^r \Psi_{\tns \uz}(a_i) \pz^i, \qquad \mbox{for $a_i \in \cUm$ and $r \in \Z$}.
$$
We have $\Ber \big(\cLmn(\uz) \big) =\ov\Psi_{\tns \uz} \left( \Ber \big(\Tmn \big) \right)$.
By \eqref{Ber-tau},
\begin{equation} \label{Ber-exp}
\Ber \big(\cLmn(\uz) \big) =\sum_{k=-\infty}^{m-n}b_k (z)\pz^k,
\end{equation}
where $b_k(z):= \Psi_{\tns \uz}(b_k)  \in \Ulz$.
\propref{generator} implies the following (cf. \cite[Proposition 4.4]{HM}).

\begin{prop} \label{Bmn-gen}
The Gaudin algebra $\Bmn$ is generated by the coefficients of the series $b_k(z)$ for $k \in \Z$ with $k \le m-n$.
\end{prop}

We also call the series $b_k(z)$ Gaudin Hamiltonians for $\glmn$.
For any $\Bmn$-module $V$, the image of $\Bmn$ in $\End(V)$ is called the {\em Gaudin algebra} of $V$ and is denoted by $\Bmn_V$. Thanks to \propref{Bmn-gen}, to study the algebra $\Bmn_V$, it suffices to examine the action of the Berezinian $\Ber \left(\cLmn(\uz) \right)$ on $V$.

Similar to \propref{perm-tau} and \propref{perm-tau-2}, we obtain the following.

\begin{prop} \label{perm}
For each $\si \in \Smn$,
\begin{eqnarray*}
\Ber^{\smn^\si} \! \left( 1+u\cLmn(\uz)^\si \right)\ns \nns &=& \nns \Ber \big(1+u\cLmn(\uz) \big),
\\
\Ber^{\smn^\si} \! \left( \cLmn(\uz)^\si \right)\ns \nns &=& \nns \Ber \big(\cLmn(\uz) \big).
\end{eqnarray*}
\end{prop}

We will again drop the symbol $|0$ from the subscript $m|0$. For example, $\cL_m(\uz):=\cL_{m|0}(\uz)$, $\Bm:=\fB_{m|0}(\uz)$, etc.
The Gaudin algebra $\fB_{m}(\uz)$ of $\gl_m$ is determined by $\cdet(\cL_{m}(\uz))$ in view of \propref{CFR}. This matches the original definition in the non-super setting (\cite{MTV06}).

\subsection{Polynomial modules}

In this subsection, we summarize some basic properties of polynomial modules.
A partition $\la=(\la_1,\la_2,\ldots)$ is called an {\em $(m|n)$-hook partition} if $\la_{m+1} \le n$. Let $\Pmn$ be the set of $(m|n)$-hook partitions.
For any $\la \in \Pmn$, let
$$\ovlamn=\sum_{i=1}^{m}\la_i\ep_i+\sum_{i=1}^{n} \left\langle \la_i^\prime-m \right\rangle\ep_{i-\hf}\in \hmn^*$$
and
$$
\Xmn=\setc*{\! \ovlamn \in \hmn^*}{\la \in \Pmn \!},
$$
whose elements are dominant weights with respect to $\bmn$.
Hereafter, $\langle r \rangle:={\rm max}\{r, 0 \}$ and $\la^\prime=(\la_1^\prime, \la_2^\prime, \ldots)$ denotes the conjugate partition of $\la$.

Define
\begin{equation}\label{weight}
\Ximn=\sum_{i \in \Imn} \Z_+ \ep_i.
\end{equation}
For $\xi \in \hmn^*$, let $\Lmn(\xi)$ denote the irreducible highest weight $\glmn$-module with highest weight $\xi$ with respect to the standard Borel subalgebra $\bmn$.
A $\glmn$-module $M$ is called a {\em polynomial module} if $M$ is $\fh_{\mn}$-semisimple and every weight of $M$ belongs to $\Ximn$.
Let
\begin{equation*}
\Ximn(\ov\varepsilon):=
\setc[\Bigg]{\! \mu\in \Ximn }{ \sum_{i \in \Imn \cap\hf+\Z_+}\mu(\Eii)\equiv \varepsilon \,\,(\text{mod }2) \!},\qquad \hbox{ for $\varepsilon= 0, 1.$}
\end{equation*}
By the description of the positive root system of $\glmn$, every polynomial module $M$ admits a natural $\Z_2$-gradation on $M={M}_{\ov{0}}\bigoplus {M}_{\ov{1}}$, where
$$
{M}_{\ov\varepsilon}:=\bigoplus_{\mu \in \Ximn(\ov\varepsilon)}{M}_{\mu}, \qquad \hbox{ for $\varepsilon= 0, 1.$}
$$
The $\Z_2$-gradation is compatible with the action of  $\glmn$.

Let $\Cmn$ denote the category of polynomial $\glmn$-modules. The morphisms in $\Cmn$ are $\glmn$-homomorphisms.
The following is well known (see, for example, \cite[Theorem 3.26 and Theorem 3.27]{CW}) and also the proof of \cite[Theorem 6.2.2]{Lus}).

\begin{prop}\label{Cmn}
The category $\Cmn$ is a semisimple tensor category, and each polynomial $\glmn$-module $M$ decomposes into a direct sum of irreducible $\glmn$-modules of the forms $\Lmn(\xi)$ with $\xi \in \Xmn$.
 \end{prop}

It is also well known that every irreducible $\glmn$-module in $\Cmn$ is finite-dimensional (see, for instance, \cite[Proposition 2.2]{CW}). The following lemma can be obtained easily by the weights of polynomial modules described in \eqnref{weight}.

\begin{lem}\label{weight=0}
 Let $M, N \in \Cmn$, and let $\mu$ and $\ga$ be weights of $M$ and $N$, respectively. Then
$$
(\mu+\ga)(\Eii)=0\quad\hbox{if and only if }\quad \mu(\Eii)=0\,\, \hbox{and}\,\, \ga(\Eii)=0,\qquad
\hbox{for $i\in \Imn$.}
$$
\end{lem}

For any $\si \in \fS_{m+n}$, we endow a new total order $<_\si$ on $\Imn$ given by
$$
i <_\si j \qquad \mbox{if $\quad \pi\si \pi^{-1}(i)< \pi\si \pi^{-1}(j)$.}
$$
Let $\bmn^{\si}$ denote the Borel subalgebra of $\glmn$ corresponding to the ordering $<_\si$. That is,
$$
\bmn^{\si}:=\bigoplus_{\substack{ i,j \in \mathbb{I}_{\mn}, \, i \le_\si j} }  \C E_{i, j}.
$$
The Borel subalgebras $\bmn^{\si}$ and $\bmn$ share the same Cartan subalgebra $\hmn$.
For any $\glmn$-module $M$, the {\em $\si$-singular space} of $M$ (with respect to $\bmn^\si$) is defined as
$$
M^{\si\mbox{-}\sing}=\setc*{\! v \in M }{E_{i, j}  v=0 \mbox{\, for all\, $i, j \in \Imn $ with $i<_\si j$}\!}.
$$
Any nonzero vector in $M^{\si\mbox{-}\sing}$ is called a $\si$-singular vector.
For any weight $\mu$ of $M$, we let $M^{\si\mbox{-}\sing}_\mu=M_\mu \cap M^{\si\mbox{-}\sing}$.
If $M^{\si\mbox{-}\sing}_\mu \not=0$, then $\mu$ is called a $\si$-singular weight of $M$, and $M^{\si\mbox{-}\sing}_\mu$ is called the $\si$-singular weight space of $M$ of ($\si$-singular) weight $\mu$.
If $\si$ is the identity, we set $M^\sing=M^{\si\mbox{-}\sing}$ and $M^\sing_\mu=M^{\si\mbox{-}\sing}_\mu$ and replace any ``$\si$-singular" with ``singular".

For $\xi \in \hmn^{*}$, let $\Lmn^\si(\xi)$ be the irreducible highest weight $\glmn$-module with highest weight $\xi$ with respect to $\bmn^{\si}$. Note that $\Lmn^\si(\xi)$ can be identified with $\Lmn(\eta)$ for some $\eta\in \hmn^*$ via a sequence of odd reflections (see \cite[Section 3.1]{CL10} or \cite[Section 6.3]{CW}). Moreover, $\xi$ can be written explicitly in terms of $\eta$ and vice versa. We will only show a special case for the correspondence in \propref{Lmnsi} below, which suffices for our purposes in this paper.

We introduce some more notations.
For $p \in \N$ with $p\le m$, define $\si_p \in \fS_{m+n}$ by
\begin{equation} \label{sip}
\si_p(i)=
\begin{cases}
i,\quad\text{if } i=1, \ldots, p;\\
i+n,\quad\text{if } i=p+1, \ldots, m;\\
i-(m-p),\quad\text{if } i=m+1, \ldots, m+n.
\end{cases}
\end{equation}
Thus, $\s_{\mn}^{\si_p}=(0^p, 1^n, 0^{m-p})$ if $\s_{\mn}:=(0^{m}, 1^n) \in \cS_{\mn}$.
For any $\la=(\la_1, \la_2, \ldots) \in \Pmn$, let
$$
 {\ovla}^{\si_p}:=\sum_{i=1}^{p}\la_i\ep_i+\sum_{i=1}^{n} \left\langle \la_i^\prime-p \right\rangle\ep_{i-\hf}+\sum_{i=p+1}^{m} \left\langle \la_{i}-n  \right\rangle \ep_i  \in \hmn^*.
$$
We define
$${\mf X }^+_{p|n|m-p}=\setc*{\! \ovla^{\si_p} \in \hmn^*}{{\la \in \Pmn}\!}\!,$$
whose elements are dominant weights with respect to $\bmn^{\si_p}$.

The following proposition follows from \propref{Cmn} together with the proof of \cite[Lemma 3.2]{CL10} (cf. \cite[Proposition 6.4]{CLW15}).

\begin{prop}\label{Lmnsi} Every irreducible $\glmn$-module in $\Cmn$ is of the form $\Lmn^{\si_p}(\xi)$ for some $\xi \in {\mf X}^+_{p|n|m-p}$.
Moreover, $\Lmn^{\si_p}(\ovla^{\si_p})$ can be identified with $\Lmn(\ovlamn)$ for any $\la \in \Pmn$, and there is an element $A_{\ovlamn}\in U(\glmn)$ such that $A_{\ovlamn}v$ is a $\si_p$-singular vector in $\Lmn^{\si_p}(\ovla^{\si_p})$ for any singular vector $v$ in $\Lmn(\ovlamn)$.
\end{prop}

For $p,k \in \Z_+$ with $m \ge p>0$ and $n \ge k$, we may regard $\Xi_{p|k}\subseteq \Ximn$.
Define the {\em truncation functor}
$\otr^{\mn}_{p|k}: \Cmn \longrightarrow \Cpk$ by
$$
 \otr^{\mn}_{p|k}(M)=\bigoplus_{\nu\in \Xi_{p|k}}M_{\nu}, \qquad \hbox{for $M\in \Cmn$},
$$
and $\otr^{\mn}_{p|k}(f)$ is defined to be the restriction of $f$ to $\otr^{\mn}_{p|k}(M)$ for $f \in {\rm Hom}_{\Cmn}(M, N)$.
When it is clear from the context, we will write $\otr_{p|k}$ instead of $\otr^{\mn}_{p|k}$.
According to the weight space decomposition $M=\bigoplus_{\nu\in \Xi_{\mn}}M_{\nu}$, the truncation functor $\otr^{\mn}_{p|k}$ is clearly an exact functor and $\otr_{p|k}(f) \in {\rm Hom}_{\Cpk}(\otr_{p|k}(M), \otr_{p|k}(N))$ for $M, N \in \Cmn$.
By \lemref{weight=0}, $\otr_{p|k}(M\otimes N)=\otr_{p|k}(M)\otimes \otr_{p|k}(N)$ for $M, N \in \Cmn$, and hence $\otr^{\mn}_{p|k}$ is a tensor functor.

The following proposition can be proved in a way similar to \cite[Lemma 3.2]{CLW11} (cf. \cite[Proposition 6.9]{CW} and \cite[Proposition 7.5]{CLW15}). 
Evidently, every $(p|k)$-hook partition is also an $(m|n)$-hook partition.

 \begin{prop}\label{truncation}
Let $p,k \in \Z_+$ with $m \ge p>0$ and $n \ge k$. For $\la \in \Pmn$, we have
 $$
 \otr_{p|k} \big( \Lmn({\ovlamn}) \big)=\begin{cases}
       L_{p|k}({\ovla^{p|k}}),&\qquad \mbox{if $\la \in \cP_{p|k}$};\\
       0,&\qquad  \mbox{otherwise}.
     \end{cases}
 $$
   \end{prop}

 \begin{proof}
   Note that $ \otr^{\mn}_{p|k}= \otr^{m|k}_{p|k} \circ \otr^{\mn}_{m|k}$. Since $\otr^{\mn}_{m|k}$ is the functor $\tr^n_k: \ov\OO_n \longrightarrow \ov\OO_k$ in \cite[Lemma 3.2]{CLW11} restricted to the corresponding categories of polynomial modules, we see, by {\em loc. cit.}, that $\otr^{\mn}_{m|k}(\Lmn({\ovlamn}))=L_{m|k}({\ovla^{m|k}})$ if $\la \in \cP_{m|k}$ and 0 otherwise.

   We claim that $\otr^{m|k}_{p|k}(L_{m|k}\big({\ovla^{m|k}})\big)=L_{p|k}({\ovla^{p|k}})$ if $\la \in \cP_{p|k}$ and 0 otherwise.
 It suffices to consider $k=n$. By \propref{Lmnsi}, $\otr^{\mn}_{p|n}\big(\Lmn({\ovlamn}) \big)=\otr^{\mn}_{p|n}\big(\Lmn^{\si_p}({\ovla}^{\si_p}) \big)$.
 Since $\otr^{\mn}_{p|n}$ is similar to the functor $\breve{\tr}$ in \cite[Proposition 7.5]{CLW15} (with ${\bf { b}}=(0^p, 1^n)$ and $k=m-p$), we can adapt the proof of {\em loc. cit.} to establish the claim.
 \end{proof}

\begin{rem}
In general, if $\la \in \cP_{p|k}$, the highest weight $\ovlamn$ on the left hand side of the equality in \propref{truncation} does not equal the highest weight $\ovla^{p|k}=\sum_{i=1}^{p}\la_i\ep_i+\sum_{i=1}^{k} \left\langle \la_i^\prime-p \right\rangle\ep_{i-\hf}  \in \fh_{p|k}^*$ on the right hand side unless $p=m$.
\end{rem}

The following corollary is a consequence of \propref{Cmn}, \propref{Lmnsi} and \propref{truncation}.

\begin{cor}\label{sip-wt}
Let $p \in \N$ with $p\le m$. For any $M \in \Cmn$ and $\la \in \cP_{p|n}$, we have
 $$
 \otr_{p|n}(M )^\sing_{\ovla^{p|n}}=M^{{\si_p}\mbox{-}\sing}_{\ovla^{\si_p}}.
 $$
\end{cor}

\subsection{Tensor products of polynomial modules} \label{tensor}

Fix $\uz\in \Xl$, $M_1, \ldots, M_\ell \in \Cmn$
and the tensor product of polynomial modules
$$
\bn M :=M_1 \otimes \cdots \otimes  M_\ell .
$$
The Gaudin algebra $\Bmn$ acts on $\bn M$ and commutes with the action of $\glmn$ (cf. \cite[Sections 3.1--3.2]{MR}).
Therefore  $\bn M^{\si\mbox{-}\sing}$ and $\bn M^{\si\mbox{-}\sing}_\mu$ are $\Bmn$-modules for any $\si \in \Smn$ and any $\si$-singluar weight $\mu$ of $\bn M$.

By \propref{Cmn} and \propref{Lmnsi} as well as the fact that $\Bmn$ commutes with the element $A_{\ovlamn}\in U(\glmn)$ in \propref{Lmnsi}, we obtain the following proposition.

\begin{prop} \label{B-isom}
Let $p \in \N$ with $p\le m$.
For any $\la \in \Pmn$, there is an isomorphism of $\Bmn$-modules
$$
\phi_{\ovlamn}^p:  \bn M ^\sing_{\ovlamn} \longrightarrow \bn M^{{\si_p}\mbox{-}\sing}_{\ovla^{\si_p}},
$$
defined by $\phi_{\ovlamn}^p (v)=A_{\ovlamn} v$ for $v \in  \bn M ^\sing_{\ovlamn}$.
\end{prop}

\begin{rem} \label{rem:B-isom}
There exists $\bar A_{\ovlamn}\in U(\glmn)$ such that the inverse of $\phi_{\ovlamn}^p$ is given by $\Big(\phi_{\ovlamn}^p\Big)^{\! -1} (w)=\bar A_{\ovlamn} w$ for $w \in \bn M^{{\si_p}\mbox{-}\sing}_{\ovla^{\si_p}}$.
\end{rem}

The following proposition plays an important role in proving our main results in \secref{main}.
Recall $\si_p \in \fS_{m+n}$ in \eqref{sip} and the action of $\Ber \big(\cLmn(\uz) \big)$ defined in \eqref{pdiff}. Also, $ \otr_{p|n}(M )^\sing_{\ovla^{p|n}}=M^{{\si_p}\mbox{-}\sing}_{\ovla^{\si_p}}$ for $\la \in \cP_{p|n}$ and $\otr_{m|k}(M)^{\sing}_\mu=M^{\sing}_\mu$  for $\mu\in \Xmn$ with $\mu(E_{k+\hf, k+\hf})=0$ by \corref{sip-wt}.

\begin{prop} \label{trunc}
Let $p,k \in \Z_+$ with $m \ge p>0$ and $n \ge k$.

\begin{enumerate} [\normalfont(i)]

\item If $\mu$ is a $\si_p$-singular weight of $\bn M$ with $\mu(\Eii)=0$ for $i=p+1,\ldots, m$ (equivalently, $\mu=\ovla^{\si_p}$ for some $\la \in \cP_{p|n}$), then
$$
\Ber \big(\cLmn(\uz) \big) v= \Ber \big(\cL_{p|n}(\uz) \big) \pz^{m-p} v, \qquad \hbox{for all $v \in \bn M^{{\si_p}\mbox{-}\sing}_{\! \mu}$.}
$$
Consequently,
$$
\Bmn_N=\fB_{p|n}(\uz)_N,
$$
where $N:=\bigoplus_\mu\otr_{p|n}(\bn M)^{\sing}_\mu=\bigoplus_\mu \bn M^{{\si_p}\mbox{-}\sing}_{\! \mu}$.
The direct sum is taken over all $\si_p$-singular weights $\mu$ of $\bn M$ with $\mu(\Eii)=0$ for $i=p+1, \ldots, m$.

\item If $\mu$ is a singular weight of $\bn M$ such that $\mu(E_{k+\hf, k+\hf})=0$, then
$$
\Ber \big(\cLmn(\uz) \big) v=  \Ber \big(\cL_{m|k}(\uz) \big)  \pz^{k-n} v, \qquad \hbox{for all $v \in \bn M^{\sing}_{\! \mu}$.}
$$
Consequently,
$$
\Bmn_{N^\prime}=\fB_{m|k}(\uz)_{N^\prime},
$$
where $N^\prime:=\bigoplus_\mu \otr_{m|k}(\bn M)^{\sing}_\mu=\bigoplus_\mu \bn M^{\sing}_{\! \mu}$.
The direct sum is taken over all singular weights $\mu$ of $\bn M$ with $\mu(E_{k+\hf, k+\hf})=0$.
\end{enumerate}
\end{prop}

\begin{proof}
It suffices to prove (i) for the case $p=m-1$ by induction. Let $N=\bn M^{\si_p\mbox{-}\sing}_{\! \mu}$. For simplicity, we suppress $\uz$.
We may write
$$
 \cL_{\mn}^{\si_p}=\begin{bmatrix}  \cL_{(m-1)|n} & X\\ Y&  \pz-E_{m,m}(z)  \end{bmatrix},
$$
where $X=\big[ (-1)^{2i+1} E_{i,m}(z) \big]_{i \in \mathbb{I}_{(m-1)|n}}$ and $Y=\big[ - E_{m, i}(z) \big]_{i \in \mathbb{I}_{(m-1)|n}}$ are respectively $(m+n-1)\times 1$ and $1 \times (m+n-1)$ matrices.
By \propref{perm}, $\Ber(\cLmn) \! = \Ber^{\smn^{\si_p}} (\cLmn^{\si_p})$, and hence by \propref{HM1}, we have
\begin{eqnarray}\label{e1}
\Ber(\cLmn) \! =\! \Ber(\cL_{(m-1)|n})  \big( \pz-E_{m,m}(z)- Y \cL_{(m-1)|n}^{-1} X \big).
\end{eqnarray}
Evidently, $E_{i,m}(z)$ acts trivially on $N$ for any $i \in \mathbb{I}_{(m-1)|n}$ because $i<_{\si_p} m$ for all $i \in \mathbb{I}_{(m-1)|n}$.
Since $E_{m,m}(z)$ also acts trivially on $N$ by the assumption on $\mu$,
we deduce that $E_{m,m}(z) + Y \cL_{(m-1)|n}^{-1} X $ acts trivially on $N$.
This completes the proof of the first part of (i) by \eqref{e1}.
The second part follows immediately as the actions of the coefficients of $\Ber(\cLmn)$ and $\Ber (\cL_{p|n}) \pz^{m-p}$ on $\bn M$ are the same.

The proof of (ii) is similar, and it is sufficient to consider the case $k=n-1$. By \propref{HM1},
$$
\Ber(\cLmn) \! =\! \Ber(\cL_{m|(n-1)})  \big( \pz+E_{n-\hf,n-\hf}(z)- Y^\prime \cL_{m|(n-1)}^{-1} X^\prime \big)^{-1},
$$
where $X^\prime=\big[ (-1)^{2i+1} E_{i,n-\hf}(z) \big]_{i \in \mathbb{I}_{m|(n-1)}}$ and $Y^\prime=\big[E_{n-\hf, i}(z) \big]_{i \in \mathbb{I}_{m|(n-1)}}$ are respectively $(m+n-1)\times 1$ and $1 \times (m+n-1)$ matrices.
Equivalently, we have
\begin{equation}\label{e2}
\Ber(\cLmn)   \big( \pz+E_{n-\hf,n-\hf}(z)- Y^\prime \cL_{m|(n-1)}^{-1} X^\prime \big) \! =\! \Ber(\cL_{m|(n-1)}).
\end{equation}
Using an argument similar to that of (i), we see that $E_{n-\hf,n-\hf}(z) - Y^\prime  \cL_{m|(n-1)}^{-1} X^\prime$ acts trivially on $N^\prime$.
This proves the first part of (ii) by \eqref{e2}. The second part is now immediate.
\end{proof}

\section{Main results} \label{main}

In this section, we start by collecting some basic facts about Frobenius algebras and the Gaudin algebras of general linear Lie algebras.
We then prove our main results.
Let $M$ be an $\ell$-fold tensor product of irreducible polynomial modules over $\glmn$.
We show that ${M^\sing}$ is a cyclic $\Bmn$-module and the Gaudin algebra $\Bmn_{M^\sing}$ of $M^\sing$ is a Frobenius algebra, where $\uz \in \Xl$.
We also show that $\Bmn_{M^\sing}$ is diagonalizable with a simple spectrum if $\uz$ is generic.

\subsection{Frobenius algebras}

Let $\cA$ be a commutative associative unital algebra and $V$ a finite-dimensional $\cA$-module. We denote by $\cA_{V}$ the image of $\cA$ in $\End(V)$. Let $\ga: \cA  \longrightarrow \C$ be a character. We consider the vector spaces
$$
E_{\cA_V}(\ga)=\setc*{\! v \in V}{ av=\ga(a) v  \,\,\,  \mbox{for all $a \in \cA$} \!}
$$
and
$$
G_{\cA_V}(\ga)=\setc*{\! v \in V}{\mbox{for all $a \in \cA$, there exists $k \in \N$ such that } \big( a-\ga(a)1 \big)^{\! k} v=0 \!}.
$$
If $E_{\cA_V}(\ga) \not=0$, then we call $\ga$ an eigenvalue of $\cA_V$ and $E_{\cA_V}(\ga)$ the eigenspace of $\cA_V$ corresponding to $\ga$. Meanwhile, any nonzero vector of $E_{\cA_V}(\ga)$ is called an eigenvector.
If $G_{\cA_V}(\ga) \not=0$, then $G_{\cA_V}(\ga)$ is called the generalized eigenspace of $\cA_V$ corresponding to $\ga$. Clearly, $G_{\cA_V}(\ga)$ is an $\cA$-module.

A  finite-dimensional commutative associative unital algebra $\cA$ is called a {\em Frobenius algebra} if there is a nondegenerate symmetric bilinear form $(\cdot, \cdot)$ on $\cA$ such that
$$
(ab, c)=(a, bc) \quad \mbox{for all $a, b, c \in \cA$.}
$$
We recall two useful lemmas from \cite{Lu20}.

\begin{lem} [{\cite[Lemma 2.7]{Lu20}}] \label{Lu20-1}
Suppose $V$ is a cyclic $\cA$-module and admits a nondegenerate symmetric bilinear form $\langle \cdot, \cdot \rangle$ with respect to which $\cA$ is symmetric (i.e., $\langle av, w \rangle=\langle v, aw \rangle$ for all $a \in A$ and $v, w \in V$), then $\cA_V$ is a Frobenius algebra.
\end{lem}

 \begin{lem} [{\cite[Lemma 1.3]{Lu20}}] \label{Lu20-2}
Suppose $V$ is a cyclic $\cA$-module and $\cA_V$ is a Frobenius algebra. Then:

\begin{enumerate} [\normalfont(i)]

\item $\cA_V$ is a maximal commutative subalgebra of $\End(V)$ of dimension $\dim (V)$.

\item  Every eigenspace of $\cA_V$ is one-dimensional, and the set of eigenspaces of $\cA_V$ is in bijective correspondence with the set of maximal ideals of $\cA_V$.

\item Every generalized eigenspace of $\cA_V$ is a cyclic $\cA$-module.

\end{enumerate}

\end{lem}

\begin{rem}
In \cite{Lu20}, the $\cA$-module $V$ which satisfies the hypothesis of \lemref{Lu20-2} is said to be perfectly integrable.
\end{rem}

\subsection{The Gaudin algebra of $\glm$}

Fix $\uz:=(z_1, \ldots, z_\ell) \in \Xl$.
Let $V_1, \ldots, V_\ell$ be finite-dimensional irreducible $\glm$-modules and
\begin{equation} \label{dominant}
\bn V =V_1  \otimes \cdots \otimes  V_\ell .
\end{equation}
Note that $V_i$'s are highest weight $\glm$-modules with dominant integral highest weights. For each $i=1, \ldots, \ell$, let $v_i$ be a highest weight vector of $V_i$.
The {\em Shapovalov form} $S_i$ is the unique nondegenerate symmetric bilinear form on $V_i$ defined by
$$
S_i(v_i, v_i)=1 \quad {\rm and} \quad S_i(E_{r,s} v, w)=S_i(v, E_{s,r} w)
$$
for all $v, w \in V_i$ and $r, s=1, \ldots, m$ (\cite{Sh}). If $v$ and $w$ are weight vectors of $V_i$ of distinct weights, then $S_i(v, w)=0$.
The bilinear forms $S_1, \ldots, S_\ell$ induce a nondegenerate symmetric bilinear form
\begin{equation} \label{Sh}
S:=S_1 \otimes \cdots \otimes S_\ell
\end{equation}
on $\bn V $, which is called the {\em tensor Shapovalov form}. The restriction of $S$ to the singular space $\bn V^\sing$ is nondegenerate as well.
By \cite[Theorem 9.1]{MTV06}, $\Bm$ is symmetric with respect to $S$.

 The following is a consequence of \cite[Main Theorem]{Ryb}.

 \begin{thm}  \label{glm-cyclic}
For $\uz \in \Xl$, $\bn V^\sing$ is a cyclic $\Bm$-module.
\end{thm}

\begin{rem} \label{rem-Frob}
The Gaudin algebra $\Bm_{\bn V^\sing}$ is a Frobenius algebra by \lemref{Lu20-1}, and hence Properties (i)--(iii) in \lemref{Lu20-2} are satisfied automatically.
\end{rem}

The diagonalization of $\Bm_{\bn V^\sing}$ is obtained as a consequence of \cite[Main Corollary]{Ryb} (see also \cite{MTV09-2, MTV09-3}).

\begin{thm} \label{glm-diag}
For a generic $\uz \in \Xl$, $\Bm_{\bn V^\sing}$ is diagonalizable with a simple spectrum.
\end{thm}

\subsection{Proofs of the main results} \label{pf}
Fix $\uz \in \Xl$ and irreducible modules $L_1, \ldots, L_\ell \in \Cmn$. Let
$$
\un L :=L_1 \otimes \cdots \otimes  L_\ell .
$$

Any singular weights of $\un L $, $L_1, \ldots, L_\ell$ are of the forms $\ovlamn$ for some $\la \in \Pmn$. There are only finitely many of them since $\un L$ is finite-dimensional. Choose $r \in \Zp$ large enough that $l(\la) \le m+r$ for all such $(m|n)$-hook partitions $\la$, where $l(\la)$ denotes the length of $\la$.
For each $i=1, \ldots, \ell$, there is an  irreducible module $\mr{L}_i$ in $\Cmrn$ such that $\otr_{\mn}(\mr{L}_i)=L_i$ by \propref{truncation}.
Let $\mr{\un L}$ be the $\glmrn$-module defined to be
\begin{equation} \label{mrL}
\mr{\un L}:=\mr{L}_1  \otimes \cdots \otimes \mr{L}_\ell.
\end{equation}

Set
\begin{equation} \label{wt}
W=\setc*{\! \la \in \Pmn}{ \mbox{$ \ovlamn$ is a singular weight of  $\un L$}\!}.
\end{equation}
For any $\la \in W$, let
$$
\phi_{\ovla^{\mrn}}^m:  \mr{\un L}^\sing_{\ovla^{\mrn}} \longrightarrow \mr{\un L}^{{\si_m}\mbox{-}\sing}_{\ovla^{\si_m}}
$$
be the $\fB_{\mrn}$-module isomorphism defined as in \propref{B-isom}, where $\si_m \in \fS_{m+r+n}$ is defined as in \eqref{sip} such that $\s_{\mrn}^{\si_m}=(0^m, 1^n, 0^r)$ if $\s_{\mrn}:=(0^{m+r}, 1^n)$.
By \corref{sip-wt},
\begin{equation} \label{sim-wt}
\mr{\un L}^{{\si_m}\mbox{-}\sing}_{\ovla^{\si_m}}=\otr_{m|n}(\mr{\un L})^\sing_{\ovlamn}.
\end{equation}
Moreover, $\ovla^{\mrn}(\Eii)=0$ for all $i\in \mathbb{I}_{\mrn}$ with $i> m+r$ by our choice of $r$.
In other words, $\ovla^{\mrn}$ can be regarded as a weight in ${\mf X}^+_{\mrz}$, and
 \begin{equation} \label{even-wt}
 \mr{\un L}^\sing_{\ovla^{\mrn}}=\otr_{(m+r)|0}(\mr{\un L})^\sing_{\ovla^{\mrn}}.
\end{equation}
Note that for each $i=1, \ldots, \ell$, $\otr_{(m+r)|0} \big(\mr{L}_i \big)=L_{(m+r)|0}(\eta_i)$ for some $\eta_i \in {\mf X}^+_{\mrz}$.

\begin{thm}\label{cyc+frob}

For $\uz \in \Xl$, we have:

\begin{enumerate}[\normalfont(i)]
\item $\un L^\sing$ is a cyclic $\Bmn$-module.

\item $\Bmn_{\un L^\sing}$ is a Frobenius algebra.
\end{enumerate}

\end{thm}

\begin{proof}
The above discussion shows that
 $$
 \un L^\sing=\bigoplus_{\la\in W}\otr_{\mn}(\mr{\un L})^\sing_\ovlamn
 =\bigoplus_{\la\in W}\mr{\un L}^{{\si_m}\mbox{-}\sing}_{\ovla^{\si_m}}.
 $$
To prove (i), it is enough to show that $\bigoplus_{\la\in W}\mr{\un L}^{{\si_m}\mbox{-}\sing}_{\ovla^{\si_m}}$ is a cyclic $\fB_{(m+r)|n}$-module by \propref{trunc}(i).
Also, this holds if $\bigoplus_{\la\in W}\mr{\un L}^{\sing}_{\ovla^{\mrn}}$ is a cyclic $\fB_{\mrn}$-module via the isomophism $\disp{\phi:=\bigoplus_{\la \in W} \phi_{\ovla^{\mrn}}^m}$.
According to \thmref{glm-cyclic},
$$
\otr_{(m+r)|0}(\mr{\un L})^\sing=\Big(\bigotimes_{i=1}^\ell \otr_{(m+r)|0} \big(\mr{L}_i \big) \Big)^\sing
$$
 is a nonzero cyclic $\fB_{m+r}$-module, and hence its direct summand
 $$N:=\disp{\bigoplus_{\la\in W} \Big(\bigotimes_{i=1}^\ell \otr_{(m+r)|0} \big(\mr{L}_i \big)   \Big)^\sing_{\ovla^{\mrn}}}$$
 is also a cyclic $\fB_{m+r}$-module as $\fB_{m+r}$ preserves singular weight spaces. By \propref{trunc}(ii), we see that $\disp{\bigoplus_{\la \in W}\mr{\un L}^\sing_{\ovla^{\mrn}}}$ is a cyclic $\fB_{m+r|n}$-module.
 This completes the proof of (i).

To prove (ii), note that the tensor Shapovalov form $S$ on $\otr_{(m+r)|0}(\mr{\un L})^\sing$ (cf. \eqref{Sh}) restricts to a nondegenerate symmetric bilinear form on $N$ with respect to which $\fB_{m+r}$ is symmetric.
By \eqref{sim-wt} and \eqref{even-wt}, we define
$$
\langle v, w \rangle=S \big(\phi^{-1} v, \phi^{-1} w \big), \qquad \mbox{for $v, w \in \un L^\sing$}.
$$
Clearly, $\langle \cdot, \cdot \rangle$ is a nondegenerate symmetric bilinear form on $\un L^\sing$.
By \propref{trunc}, $\Bmn$ is also symmetric with respect to $\langle \cdot, \cdot \rangle$. This proves (ii) in view of \lemref{Lu20-1}.
\end{proof}

\begin{cor} \label{max}
For $\uz \in \Xl$, the following properties hold:

\begin{enumerate} [\normalfont(i)]

\item The algebra $\Bmn_{\un L^\sing}$ is a maximal commutative subalgebra of $\End(\un L^\sing)$ of dimension $\dim \! \left({\un L^\sing} \right)$.

\item Every eigenspace of the algebra $\Bmn_{\un L^\sing}$ is one-dimensional, and the set of eigenspaces of $\Bmn_{\un L^\sing}$ is in bijective correspondence with the set of maximal ideals of $\Bmn_{\un L^\sing}$.

\item Every generalized eigenspace of $\Bmn_{\un L^\sing}$ is a cyclic $\Bmn$-module.

\end{enumerate}
\end{cor}

\begin{proof}
This follows from \thmref{cyc+frob} and \lemref{Lu20-2}.
\end{proof}

Recall the expansion $\disp{\Ber \big(\cLmn(\uz) \big) = \sum_{k=-\infty}^{m-n}b_k (z)\pz^k}$ in \eqref{Ber-exp}.
Let $V$ be a $\Bmn$-module and $v$ an eigenvector of $\Bmn_V$. Then there exists $\al_k(z) \in \C \blb z^{-1} \brb$ such that $b_k(z) v=\al_k(z) v$ for $k \le m-n$.
We call
\begin{equation}\label{diff}
\cD_{\tns v}:= \sum_{k=-\infty}^{m-n} \al_k(z)  \pz^k
\end{equation}
the scalar differential operator associated to $(\Ber \big(\cLmn(\uz) \big), v)$.
Clearly, $\Ber \big(\cLmn(\uz) \big)  v=\cD_{\tns v} v$.
For $n=0$, $\cdet(\cL_m(\uz))  v=\cD_{\tns v} v$ by \propref{CFR}.

We will answer affirmatively the diagonalization of the action of $\Bmn$ on the singular space of any $\ell$-fold tensor product of irreducible polynomial $\glmn$-modules. To achieve this, we need the following proposition concerning eigenbases.
For any $\la \in W$, let $\ga=\ovlamn$ and $\mr \ga=\ovla^{\mrn}$. Recall that $r$ is chosen such that $l(\la) \le m+r$.

\begin{prop} \label{eigen}

\begin{enumerate} [\normalfont(i)]

\item Suppose that ${\rm B}$ is an eigenbasis for $\fB_{m+r}(\uz)_{{\mr{\un L}}^\sing_{\mr \ga}}$. Then $\phi_{\mr \ga}^m ({\rm B})$ is an eigenbasis for $\Bmn_{\un L^\sing_{\ga}}$.
Moreover, for any $v \in {\rm B}$,
$$
\Ber \big(\cLmn(\uz) \big)  \big( \phi_{\mr \ga}^m (v) \big) = \cD_v \pz^{-n-r} \big( \phi_{\mr \ga}^m (v) \big),
$$
where $\cD_v$ is the scalar differential operator associated to $\big(\cdet(\cL_{m+r})(\uz), v \big)$.

\item Suppose that ${\rm B}^\prime$ is an eigenbasis for $\Bmn_{\un L^\sing_{\ga}}$. Then $(\phi_{\mr \ga}^m \big)^{-1}({\rm B}^\prime)$ is an eigenbasis for $\fB_{m+r}(\uz)_{{\mr{\un L}}^\sing_{\mr \ga}}$.
Moreover, for any $w \in {\rm B}^\prime$,
$$
\cdet(\cL_{m+r}(\uz)) \big((\phi_{\mr \ga}^m)^{-1} (w) \big)  = \cD_w^\prime \pz^{n+r} \big( (\phi_{\mr \ga}^m)^{-1} (w) \big),
$$
where $\cD_w^\prime$ is the scalar differential operator associated to $\big(\Ber \big(\cLmn(\uz) \big), w \big)$.

\end{enumerate}
\end{prop}

\begin{proof}
We will only prove (i). The proof of (ii) is similar.
Let ${\rm B}$ be an eigenbasis for $\fB_{m+r}(\uz)_{{\mr{\un L}}^\sing_{\mr \ga}}$ and $v \in {\rm B}$.
Again, we drop $\uz$.
We have
$\Ber(\cL_{m+r}) v =\cdet(\cL_{m+r}) v= \cD_v v.$
Recall that $\phi_{\mr \ga}^m (v)= A_{\mr \ga} v$ for some $A_{\mr \ga} \in U(\gl_{m+r|n})$.
Using \propref{B-isom} and \propref{trunc}, we find that
\begin{eqnarray*}
\Ber(\cLmn) \big( \phi_{\mr \ga}^m (v) \big) \ns \nns &=& \nns \Ber(\cL_{\mrn}) \, \pz^{-r} \big( A_{\mr \ga} v \big) \\
\nns &=& \nns A_{\mr \ga} \,  \Ber(\cL_{\mrn}) \, \pz^{-r} v \\
\nns &=& \nns A_{\mr \ga} \, \Ber(\cL_{m+r})\,  \pz^{-n-r} v \\
\nns &=& \nns A_{\mr \ga} \,  \cD_v \pz^{-n-r} v \\
\nns &=& \nns \cD_v \pz^{-n-r}  \big(\phi_{\mr \ga}^m (v) \big).
\end{eqnarray*}
This also means that $\phi_{\mr \ga}^m (v) $ is an eigenvector for $\Bmn_{\un L^\sing_{\ga}}$, and (i) follows.
\end{proof}

 \begin{thm} \label{diag}
For a generic $\uz \in \Xl$, $\Bmn_{\un L^\sing}$ is diagonalizable with a simple spectrum.
\end{thm}

\begin{proof}
 The diagonalization of $\Bmn_{\un L^\sing}$, for a generic $\uz \in \Xl$, follows from \thmref{glm-diag} and \propref{eigen}, and the property of having a simple spectrum is an immediate consequence of \corref{max}.
 \end{proof}

\begin{rem}
In \cite{CCL}, the cubic Gaudin Hamiltonians for $\glmn$ on $\un L^\sing$ are introduced, and the diagonalization of the Hamiltonians is established. This result is a special case of \thmref{diag}.
\end{rem}

\section{The completeness of the Bethe ansatz and the Feigin--Frenkel center}  \label{BAFF}

In this section, we apply our main results and the Bethe ansatz to obtain a set of candidates for eigenvectors and the corresponding eigenvalues for the Gaudin algebra $\Bmn_{\un L^\sing}$, where $\uz \in \Xl$ and $\un L :=L_1 \otimes \cdots \otimes  L_\ell$ for any irreducible modules $L_1, \ldots, L_\ell \in \Cmn$.

Based on the work of Mukhin, Tarasov, and Varchenko on the Gaudin algebra of the general linear Lie algebra, we give an eigenbasis for $\Bmn_{\un L^\sing}$ for a generic $\uz$ and relate the corresponding eigenvalues to the coefficients of Fuchsian differential operators. This shows that a reformulation of the Bethe ansatz is complete for $\Bmn_{\un L^\sing}$.

The Feigin--Frenkel center $\zgmn$ is conjecturally generated by a distinguished family of Segal--Sugawara vectors. We end the section with a few remarks on the conjecture.

\subsection{The Bethe ansatz} \label{BAC}

Let $\uz \in \Xl$, and let $\bn V =V_1  \otimes \cdots \otimes  V_\ell$ be as in \eqref{dominant}.
We describe the Bethe ansatz method of constructing eigenvectors for the Gaudin algebra $\Bm_{\bn V^\sing}$ (\cite{BF, FFR}).

For each $i=1, \ldots, \ell$, $V_i=L_{m|0}(\xi_i)$ for some dominant integral weight $\xi_i$.
Let $|0 \rangle=v_1 \otimes \ldots \otimes v_\ell \in \bn V $, where $v_i$ is a highest weight vector of $V_i$ for $i=1, \ldots, \ell$, and
let $f_j=E_{j+1,j}$ for $j=1, \ldots, m-1$.
Given $i_1,\ldots,i_p \in \{1,\ldots, m-1 \}$ (not necessarily distinct) and any pairwise distinct $w_1,\ldots, w_p \in \C$ with $w_j \ne z_k$ for all $j,k$,
we define
$$
 \big|w_1^{i_1},\ldots,w_p^{i_p} \big\rangle =\sum_{(I^1,\ldots,I^\ell)} \prod_{k=1}^\ell  \frac{f_{i_{j^k_1}}^{(k)} \cdots f_{i_{j^k_{a_k}}}^{(k)}} {(w_{j^k_1}-w_{j^k_2}) \cdots (w_{j^k _{a_k}}-w_{j^k _{a_k+1}})} |0 \rangle   \in \bn V .
$$
The summation is taken over all ordered partitions $I^1\cup I^2\cup \ldots \cup I^\ell$ of the set $\{1, \ldots, p \}$, where $I^k :=\big\{j^k _1,j^k _2,\ldots,j^k _{a_k } \big\}$ and $w_{j^k _{a_k+1}}:=z_k$ for $k=1, \ldots, \ell$.
The vector $\big|w_1^{i_1},\ldots,w_p^{i_p} \big\rangle$ is called a {\em Bethe vector} in $\bn V $.

Let $\{ \al_i:=\ep_i-\ep_{i+1} \, | \, i=1, \ldots, m-1 \}$ be the set of all simple roots of $\glm$, and let $\check\al_i=\Eii-E_{i+1,i+1}$ for $i=1, \ldots, m-1$.
The equations
\begin{equation} \label{BAn}
\sum_{k=1}^\ell\frac{\xi_k (\check \al_{i_j})}{w_j-z_k}-\sum_{\substack{s=1 \\ s\ne j}}^p \frac{\al_{i_s}(\check\al_{i_j})}{w_j-w_s}=0,  \qquad j=1,\ldots, p,
\end{equation}
are called the {\em Bethe ansatz equations} (\cite{BF, FFR}). If \eqref{BAn} are satisfied and the vector $\big|w_1^{i_1},\ldots,w_p^{i_p} \big\rangle$ is nonzero, then $\big|w_1^{i_1},\ldots,w_p^{i_p} \big\rangle$ is an eigenvector of $\Bm_{\bn V }$ and lies in $\bn V^\sing$ (\cite{MTV06, RV}).

The {\em completeness of the Bethe ansatz} is a conjecture that asserts that the Bethe vectors form a basis for $\Bm_{\bn V^\sing}$ for a generic $\uz$.
It is established if each of $\xi_1, \ldots, \xi_\ell$ is either the first or the last fundamental weight for $\glm$ (see \cite{MV05}).
It is, however, shown in \cite{MV07} that we may take $\ell=2$ and find highest weights $\xi_1$ and $\xi_2$ for $\gl_3$ such that the corresponding Bethe ansatz equations have no solutions for any $(z_1, z_2)\in {\bd X}_{\ns 2}$. Thus, the Bethe ansatz is incomplete for the algebra $\fB_3(z_1, z_2)_{M^\sing}$, where $M=L_{\gl_3}(\xi_1) \otimes L_{\gl_3}(\xi_2)$.

Let
$$
\cE_{\ns i}(z)=\sum_{k=1}^\ell \frac{\xi_k (\Eii)}{z-z_k}-\sum_{s=1}^p\frac{\al_{i_s}(\Eii)}{z-w_s}, \qquad i=1, \ldots, m.
$$
Suppose the Bethe ansatz equations \eqref{BAn} are satisfied. The eigenvalues of $\Bm$ acting on the Bethe vectors are determined by the formula:
\begin{equation}  \label{Bethe-glm-eigenvalue}
\cdet(\cL_m(\uz)) \big|w_1^{i_1},\ldots,w_p^{i_p} \big\rangle=\big(\pz-\cE_{\ns 1}(z) \big) \cdots \big(\pz-\cE_{\ns m}(z) \big)  \big|w_1^{i_1},\ldots,w_p^{i_p} \big\rangle.
\end{equation}
This is proved in \cite[Theorem 9.2]{MTV06}. The description here is similar to that of \cite[Theorem 3.2]{MM17} (see also \cite[Section 3.4]{MM17}).

For any irreducible modules $L_1, \ldots, L_\ell \in \Cmn$, let $\un L=L_1 \otimes \cdots \otimes  L_\ell$.
Recall the $\glmrn$-module $\mr{\un L}$ defined in \eqref{mrL} and the set $W$ of weights defined in \eqref{wt}. For any $\la \in W$, let $\ga=\ovlamn$ and $\mr \ga=\ovla^{\mrn}$.
The Bethe vectors for $\fB_{m+r}(\uz)_{\mr{\un L}^\sing_{\mr \ga}}$ induce natural candidates for eigenvectors for the algebra $\Bmn_{\un L^\sing}$. We may determine the corresponding eigenvalues by the action of the Berezinian $\Ber \big(\cLmn(\uz) \big)$.

\begin{thm} \label{Bethe-glmn-eigenvalue}

Suppose the Bethe ansatz equations \eqref{BAn} are satisfied for $\fB_{m+r}(\uz)_{\mr{\un L}^\sing_{\mr \ga}}$.
Then
\begin{eqnarray*}
& &\Ber \big(\cLmn(\uz) \big)  \left( \tns \phi_{\mr \ga}^m \big( \big|w_1^{i_1},\ldots,w_p^{i_p} \big\rangle \big) \tns \right) \\
& & \hspace{2cm} = \big(\pz-\cE_{\ns 1}(z) \big) \cdots \big(\pz-\cE_{\ns m+r}(z) \big)  \pz^{-n-r} \! \left( \tns \phi_{\mr \ga}^m \big( \big|w_1^{i_1},\ldots,w_p^{i_p} \big\rangle \big) \tns \right)
\end{eqnarray*}
provided that $\big|w_1^{i_1},\ldots,w_p^{i_p} \big\rangle$ are Bethe vectors in $\mr{\un L}^\sing_{\mr \ga}$.

\end{thm}

\begin{proof}
This follows from \propref{eigen} and the formula \eqref{Bethe-glm-eigenvalue}.
\end{proof}

\begin{rem}
If the completeness of the Bethe ansatz is valid for $\fB_{m+r}(\uz)_{\mr{\un L}^\sing_{\mr \ga}}$, then the vectors $\phi_{\mr \ga}^m \big( \big|w_1^{i_1},\ldots,w_p^{i_p} \big\rangle \big)$ form an eigenbasis for $\Bmn_{\un L^\sing_{\ga}}$.
\end{rem}

 \subsection{Fuchsian differential operators} \label{MTV}

Let $\uz \in \Xl$. For $V_1, \ldots, V_\ell \in \cC_{m|0}$, let $\bn V =V_1  \otimes \cdots \otimes  V_\ell$.
Mukhin, Tarasov, and Varchenko relate a set of eigenvectors for the Gaudin algebra $\Bm_{\bn V^\sing}$ to the Fuchsian differential operators of order $m$ with polynomial kernels and prescribed singularities. The reader is referred to \cite[Section 3.1]{MV04} for the basics of Fuchsian differential operators.

For $i=1, \ldots, m$, $V_i=L_{m|0}(\xi_i)$ for some $\xi_i \in {\mf X}^+_{m|0}$.
For $i=1,\ldots,\ell$, write $\xi_i=\sum_{j=1}^m \xi_{i,j} \ep_j$ for some $\xi_{i,j} \in \Zp$.
For any singular weight $\mu$ of $\bn V$, write $\mu=\sum_{j=1}^m \mu_j \ep_j$.
Let $\Delta_{\uxi, \mu, \uz}$ be the set of all monic Fuchsian differential operators of order $m$,
\begin{equation} \label{Fuc}
\cD:=\pz^m+\sum_{i=1}^m h_i^\cD(z) \pz^{m-i},
\end{equation}
which satisfy the following properties:

\begin{enumerate}[\normalfont(a)]

\item The singular points of $\cD$ are $z_1,\ldots,z_\ell$ and $\infty$ only.

\item For $i=1,\ldots,\ell$, the exponents of $\cD$ at $z_i$ are equal to $\xi_{i,m},\xi_{i,m-1}+1,\ldots,\xi_{i,1}+m-1$.

\item The exponents of $\cD$ at $\infty$ are equal to $1-m-\mu_1, 2-m-\mu_2,\ldots, -\mu_m$.

\item The kernel of $\cD$ consists of polynomials only.

\end{enumerate}
The set $\Delta_{\uxi, \mu, \uz}$ is nonempty only if $|\mu|=\sum_{i=1}^\ell |\xi_i|$. Here, for instance, $|\mu|:=\sum_{j=1}^m \mu_j$.
We refer to \cite{MTV09-1, MTV09-2, MTV09-3} for details on $\Delta_{\uxi, \mu, \uz}$.

Let $v$ be any eigenvector of the algebra $\Bm_{\bn V^\sing_{\! \mu}}$. We denote by $E_v$ the eigenspace containing $v$. The space $E_v$ is one-dimensional by \remref{rem-Frob}.
Let $\cD_{\tns v}$ be the scalar differential operator associated to $\big(\cdet(\cL_m(\uz)), v \big)$ (see \eqref{diff}).
According to \cite[Theorem A.1]{MTV09-1}, $\cD_{\tns v} \in \Delta_{\uxi, \mu, \uz}$.
In addition, as seen in \cite[Theorem 6.1]{MTV09-2}, the assignment $E_v \mapsto \cD_{\tns v}$ is a bijection from the set of eigenspaces of $\Bm_{\bn V^\sing_{\! \mu}}$ to the set $\Delta_{\uxi, \mu, \uz}$. We denote its inverse by $\cD \mapsto E^\cD$.
For each $\cD \in \Delta_{\uxi, \mu, \uz}$, let $\omega(\cD)$ be a nonzero element of $E^\cD$. We have
\begin{equation} \label{glm-eigenvalue}
\cdet(\cL_m(\uz)) \big( \omega(\cD) \big)=\cD \big( \omega(\cD) \big).
\end{equation}
The construction of $\omega(\cD)$ may be viewed as a reformulation of the Bethe ansatz for $\Bm_{\bn V^\sing}$.

\begin{thm} [{\cite{MTV09-2, MTV09-3}}] \label{glm-complete}
For a generic $\uz \in \Xl$, the set
$$
\setc*{\omega(\cD) \in \bn V^\sing_{\! \mu}}{\cD \in \Delta_{\uxi, \mu, \uz}}
$$
is an eigenbasis for the algebra $\Bm_{\bn V^\sing_{\! \mu}}$.

\end{thm}

Let $\un L$, $\mr{\un L}$, $\ga$ and $\mr \ga$ be as in \secref{BAC}.
Recall that for $i=1, \ldots, \ell$, $\otr_{(m+r)|0} \big(\mr{L}_i \big)=L_{(m+r)|0}(\eta_i)$ for some $\eta_i \in {\mf X}^+_{\mrz}$.
Set $\ueta=(\eta_1, \ldots, \eta_\ell)$.
We may give an eigenbasis for $\Bmn_{\un L^\sing_{\ga}}$ with a description of the corresponding eigenvalues in a more explicit form.

\begin{thm} \label{glmn-eigenbasis}

For a generic $\uz \in \Xl$, the set
$$
\setc*{\phi_{\mr \ga}^m \big( \omega(\cD) \big)}{ \cD \in \Delta_{\ueta, \mr \ga, \uz}}
$$
is an eigenbasis for the algebra $\Bmn_{\un L^\sing_{\ga}}$. Moreover, for any $\cD \in \Delta_{\ueta, \mr \ga, \uz}$,
$$
\Ber \big(\cLmn(\uz) \big) \! \left( \phi_{\mr \ga}^m \big( \omega(\cD) \big) \right)= \cD \pz^{-n-r} \! \left( \phi_{\mr \ga}^m \big( \omega(\cD) \big) \right).
$$

\end{thm}

\begin{proof}
This follows from \propref{eigen}, \thmref{glm-complete} and the formula \eqref{glm-eigenvalue}.
\end{proof}

\thmref{diag}, together with \thmref{glmn-eigenbasis}, may be considered as the completeness of a reformulation of the Bethe ansatz for the Gaudin algebra $\Bmn_{\un L^\sing}$.

 \subsection{Some remarks on the Feigin--Frenkel center} \label{c-FFC}

Recall the algebras $\zgmn$ and $\zmn$ defined in \secref{FFC}.
The superalgebra $U(\gmi)$ is equipped with the (even) derivation $\T:=-d/dt$ defined by
$$
\T(1)=0 \quad {\rm and} \quad \T(A[-r])=r A[-r-1] \quad \mbox{for $A \in \glmn$ and $r \in \N$.}
$$
The Feigin--Frenkel center $\zgmn$ is viewed as a subalgebra of $U(\gmi)$ and is $\T$-invariant (cf. the translation operator $T \in \End \big(\Vac(\glmn) \big)$ in \secref{FFC}).

We have seen that $\zmn$ is a commutative subalgebra of $\zgmn$ and gives rise to the Gaudin algebra $\Bmn$ for $\uz \in \Xl$.
There also exists a large commutative subalgebra of $\zgmn$ containing $\zmn$, which we now define.
Let $\zmnh$ be the subalgebra of $U(\gmi)$ generated by the set
$$
\setc*{\! \T^r (b_i) }{i \le m-n, \, i \in \Z, \, r \in \Zp \!},
$$
where $b_i$'s are given by \eqref{Ber-tau}.
Clearly, $\zmn \subseteq \zmnh$.
By \propref{MR} and the fact that $\zgmn$ is $\T$-invariant, we see that $\zmnh$ is a commutative subalgebra of $\zgmn$.

Again, the symbol $|0$ is dropped from the subscript $m|0$. For instance, $\zgm:=\fz(\wh{\gl}_{m|0})$, $\hat \fz_{m}:=\hat \fz_{m|0}$, etc.
By the Feigin--Frenkel theorem \cite{FF}, the algebra $\zgm$ has a complete set of Segal--Sugawara vectors $S_1, \ldots, S_m$ (see also \cite{GW, Ha}).
That is, the set $\setc*{\! \T^r(S_i) }{ i=1, \ldots, m, \, r \in \Zp  \!}$ is algebraically independent and $\fz(\widehat{\gl}_{m})=\hat \fz_{m}.$

An explicit example of a complete set of Segal--Sugawara vectors for $\zgm$ is given in \cite{CT, CM}. It can be described as follows.
We can expand $\cdet (\cT_m)$ as
$$
\cdet (\cT_m)=\sum_{i=0}^{m}  a_i \tau^i, \qquad \mbox{for some $a_i \in \zgm$.}
$$
According to \cite[Theorem 3.1]{CM} (see also \cite[Theorem 7.14]{Mo}), $\{a_1, \ldots, a_m \}$ is a complete set of Segal--Sugawara vectors.

Suppose $m, n \in \N$. Unlike the non-super case, the elements $\T^r (b_i)$, for $i \in \Z$ with $i \le m-n$ and $r \in \Zp$, are not algebraically independent (see \cite[Remark 3.4(ii)]{MR}).
However, we expect the following conjecture to be true (cf. {\em loc. cit.}).

\begin{conj} \label{FFC-gen}
For $n \in \N$, the algebra $\zgmn$ is generated by $\T^r (b_i)$ for $i \in \Z$ with $i \le m-n$ and $r \in \Zp$. In other words, $\zgmn=\zmnh$.
\end{conj}

\begin{rem}
Molev and Mukhin \cite{MM15} have proved that $\fz(\wh \gl_{1|1})=\hat \fz_{1|1}$.
\end{rem}

The algebra $\zmnh$ gives rise to a subalgebra of $\Ul$ containing the Gaudin algebra $\Bmn$.
Recall the map $\Psi_{\uz}$ defined in \secref{Gaudin}.
Let $\Bmnh$ be the subalgebra of $\Ul$ generated by the coefficients of the series $\Psi_{\uz} \big(\T^r (b_i) \big)$, for $i \in \Z$ with $i \le m-n$ and $r \in \Zp$.
It is obvious that $\Bmn$ is a subalgebra of $\Bmnh$.
The following proposition says that $\Bmnh$ and $\Bmn$ are equal.

\begin{prop} \label{der}
For $z \in \Xl$, $\Bmnh=\Bmn$.
\end{prop}

\begin{proof}
It is straightforward to verify that
$$
 \Psi_{\tns \uz}\big(\T(A_1[-r_1] \ldots A_k[-r_k]) \big)=\frac{d}{dz} \big(\Psi_{\uz} (A_1[-r_1] \ldots A_k[-r_k]) \big)
$$
for $A_1, \ldots, A_k \in \glmn$, $r_1, \ldots, r_k \in \N$ and $k \in \N$.
It follows that
$$
\hspace{3cm} \Psi_{\tns \uz} \big(\T(a) \big)=\frac{d}{dz} \big( \Psi_{\tns \uz} (a) \big) \qquad \mbox{ for all $a \in U(\gmi)$}.
$$
Thus, for all $i \in \Z$ with $i \le m-n$ and $r \in \Zp$,
$$
\Psi_{\uz} \big(\T^r (b_i) \big)=\frac{d^r}{dz^r} \tns \big( \Psi_{\tns \uz}(b_i) \big)=\frac{d^r}{dz^r} \tns \big( b_i(z) \big).
$$
Since the coefficients of the series $\disp{\frac{d^r}{dz^r} \tns \big( b_i(z) \big)}$ clearly belong to $\Bmn$, the proposition follows.
\end{proof}

Let $\un L=L_1 \otimes \cdots \otimes  L_\ell$, where $L_1, \ldots, L_\ell \in \Cmn$ are irreducible.
The algebra $\Bmnh_{\un L^\sing}$ coincides with the Gaudin algebra $\Bmn_{\un L^\sing}$, which is a maximal commutative subalgebra of $\End(\un L^\sing)$ for $\uz \in \Xl$ (\corref{max}) and is diagonalizable for a generic $\uz$ (\thmref{diag}).
The maximality gives an indication that \conjref{FFC-gen} is likely to hold.

\vskip 0.5cm
\noindent{\bf Acknowledgments.}
The first author was partially supported by NSTC grant 113-2115-M-006-010 of Taiwan.
The second author was partially supported by NSTC grant 112-2115-M-006-015-MY2 of Taiwan.

\end{document}